\documentclass[runningheads]{llncs}
\usepackage{graphicx}
\usepackage{amsmath}
\usepackage{amssymb}
\newcommand{\norm}[1]{\left\|#1\right\|}
\newcommand{\abs}[1]{\left\vert#1\right\vert}

\newcommand{\kl}[1]{\left(#1\right)}
\newcommand{\Kl}[1]{\left\{#1\right\}}

\newcommand{\R}{\mathbb{R}}

\newcommand{\grad}{\nabla}

\newcommand{\HoO}{{H^1(\Omega)}}

\newcommand{\LtO}{{L^2(\Omega)}}
\newcommand{\LiO}{{L^\infty(\Omega)}}

\newcommand{\intVx}[1]{\int\limits_\Omega #1 \, d\mathbf{x}}

\newcommand{\x}{\textbf{x}}
\newcommand{\uf}{\mathbf{u}}
\newcommand{\vf}{\mathbf{v}}

\newcommand{\RE}{\mathcal{R}}

\renewcommand{\S}{\mathcal{S}}
\newcommand{\Sb}{\bar{\mathcal{S}}}

\newcommand{\ufhi}{\hat{\uf}^i}
\newcommand{\xhi}{\hat{\x}^i}

\newcommand{\nv}{\vec{\boldsymbol{n}}}

\newcommand{\ufbg}{\uf^{\text{bg}}}
\newcommand{\ufupd}{\uf^{\text{upd}}}

\begin{document}
\title{Challenges for Optical Flow Estimates in Elastography\thanks{Supported by the Austrian Science Fund (FWF): project F6807-N36 (ES and OS), project F6805-N36 (SH), and project F6803-N36 (LK and WD).} 
}
\titlerunning{Challenges for Optical Flow Estimates in Elastography}
\author{Ekaterina Sherina\inst{1}\orcidID{0000-0002-9542-5145} \and
Lisa Krainz\inst{2}\orcidID{0000-0003-4436-8205} \and
Simon Hubmer\inst{3}\orcidID{0000-0002-8494-5188} \and
Wolfgang Drexler\inst{2}\orcidID{0000-0002-3557-6398} \and
Otmar Scherzer\inst{1,3}\orcidID{0000-0001-9378-7452}
}
\authorrunning{E.~Sherina et al.}
\institute{University of Vienna, Faculty of Mathematics, Oskar Morgenstern-Platz 1, 1090 Vienna, Austria
\email{ekaterina.sherina@univie.ac.at, otmar.scherzer@univie.ac.at}
% \\
% \url{https://csc.univie.ac.at/} 
%
\and
Medical University of Vienna, Center for Medical Physics and Biomedical Engineering, W\"ahringer G\"urtel 18-20, 1090 Vienna, Austria \email{lisa.krainz@meduniwien.ac.at, wolfgang.drexler@meduniwien.ac.at}
\and
Johann Radon Institute Linz, Altenbergerstraße 69, A-4040 Linz, Austria\\
\email{simon.hubmer@ricam.oeaw.ac.at, otmar.scherzer@ricam.oeaw.ac.at}
% \\
% \url{www.ricam.oeaw.ac.at}
}
\maketitle              % typeset the header of the contribution
\begin{abstract}
In this paper, we consider visualization of displacement fields via optical flow methods in elastographic experiments consisting of a static compression of a sample. We propose an elastographic optical flow method (EOFM) which takes into account experimental constraints, such as appropriate boundary conditions, the use of speckle information, as well as the inclusion of structural information derived from knowledge of the background material. We present numerical results based on both simulated and experimental data from an elastography experiment in order to demonstrate the relevance of our proposed approach. 

\end{abstract}

% % % % % % % % % % % % % % 
% Section - Introduction  %
% % % % % % % % % % % % % %
\section{Introduction and motivation}

The ultimate goal of elastography is to reconstruct material parameters of a sample, such as the Lam\'e parameters $\lambda, \mu$, the Young's modulus $E$, or the Possion ratio $\nu$, by exposing it to external forces. This problem is widely used in Medicine, in particular for the non-invasive identification of malignant formations inside the human skin or tissue biopsies during surgeries.

The general strategy has given rise to a number of different elastography approaches; see e.g.~\cite{Doyley_2012,Manduca_Oliphant_Dresner_Mahowald_Kruse_Amromin_Felmlee_Greenleaf_Ehman_2001,Schmitt_1998} and the references therein. These use different external forces (e.g., quasi-static, harmonic, or transient) and measure the resulting deformation either only on the boundary or everywhere inside the sample (using all kinds of imaging techniques such as e.g., X-ray, ultrasound, magnetic resonance, or optical imaging, to name but a few). In order to infer elastic material properties from these measurements, computational inversion techniques, assuming suitable material models such as linear, visco, or hyper-elasticity, have to be implemented. The most common strategy for elastography is a two-step approach consisting in first imaging the sample during displacement with the favorite imaging system and secondly by visualizing the displacement field inside the specimen from which the elastic material parameters can be computed. But also all-in-once approaches are used, which aim for visualization of material parameters directly.

\begin{figure}[t] 
    \centering
    \includegraphics[width=0.49\textwidth, clip=true, trim={6.5cm 5cm 6.5cm 6cm}]{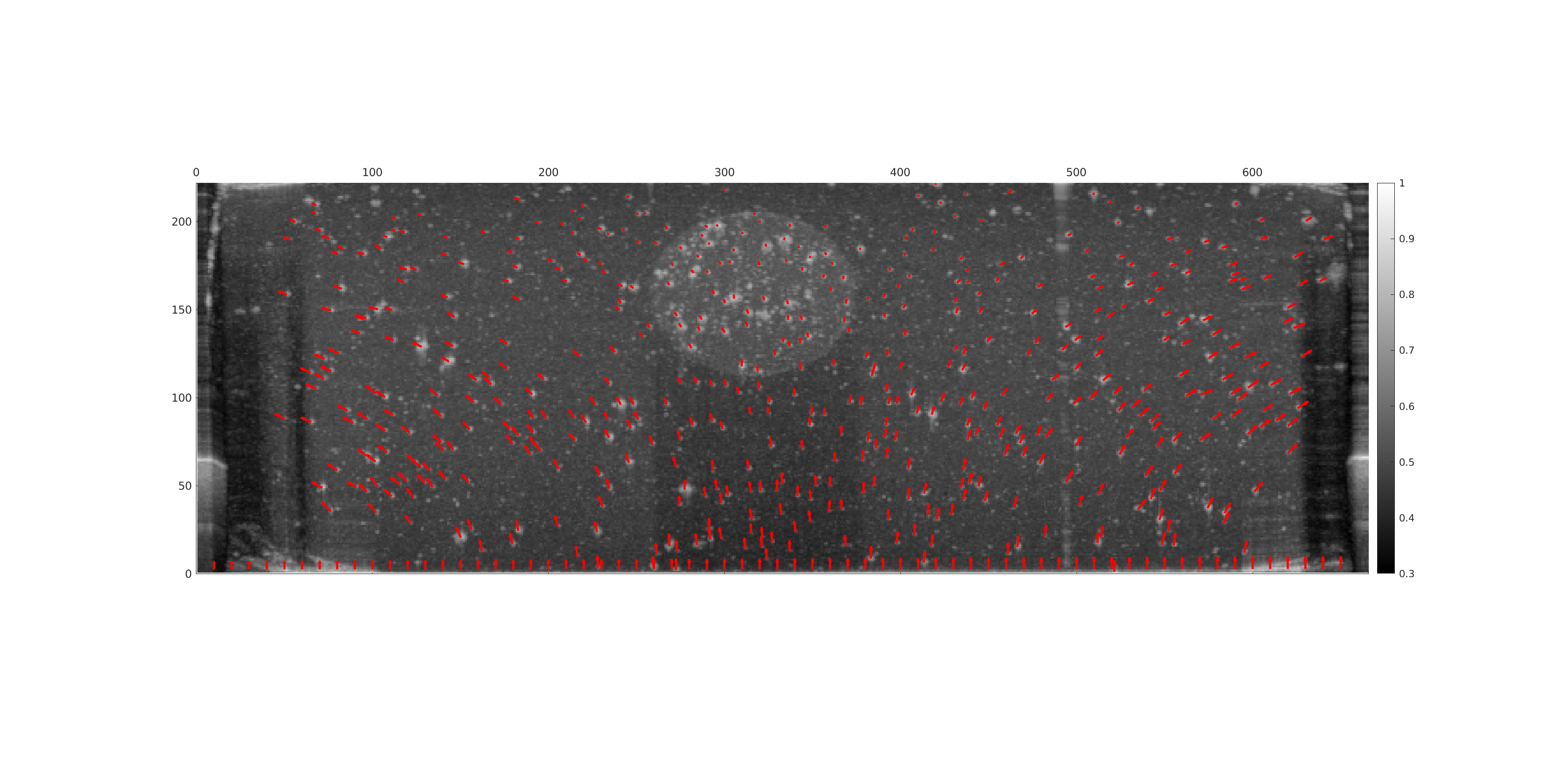}
    \includegraphics[width=0.49\textwidth, clip=true, trim={6.5cm 5cm 6.5cm 6cm}]{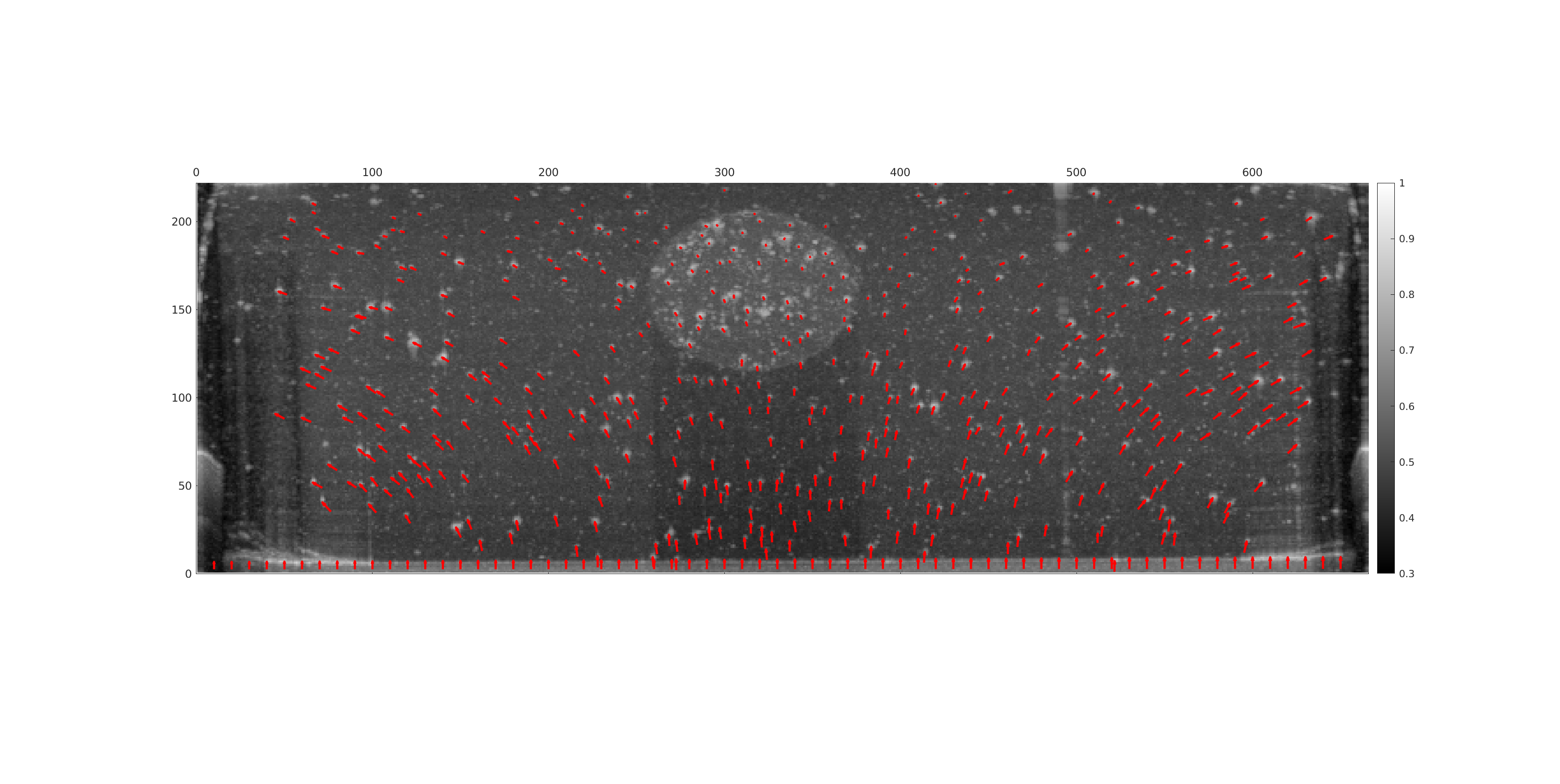}
    \caption{Example of two tomograms from a compressed sample in a quasi-static OCT elastography experiment with arrows (red) indicating the motion of speckle formations.
    }
    \label{fig_exp_oct}
\end{figure}

\begin{figure}[t] 
    \centering
    \includegraphics[width=0.49\textwidth, clip=true, trim={0cm 4cm 0cm 3cm}]{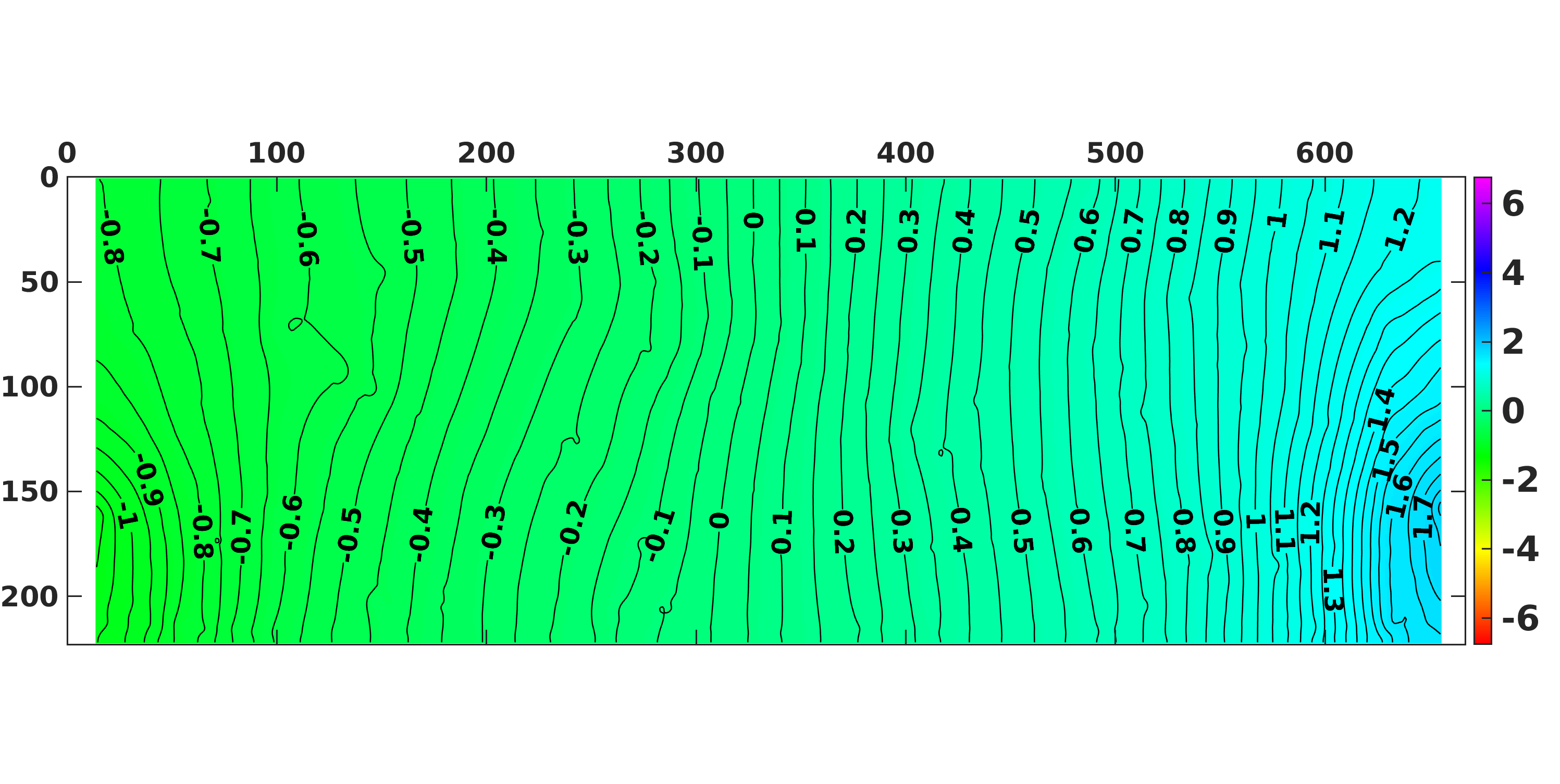}
    \includegraphics[width=0.49\textwidth, clip=true, trim={0cm 4cm 0cm 3cm}]{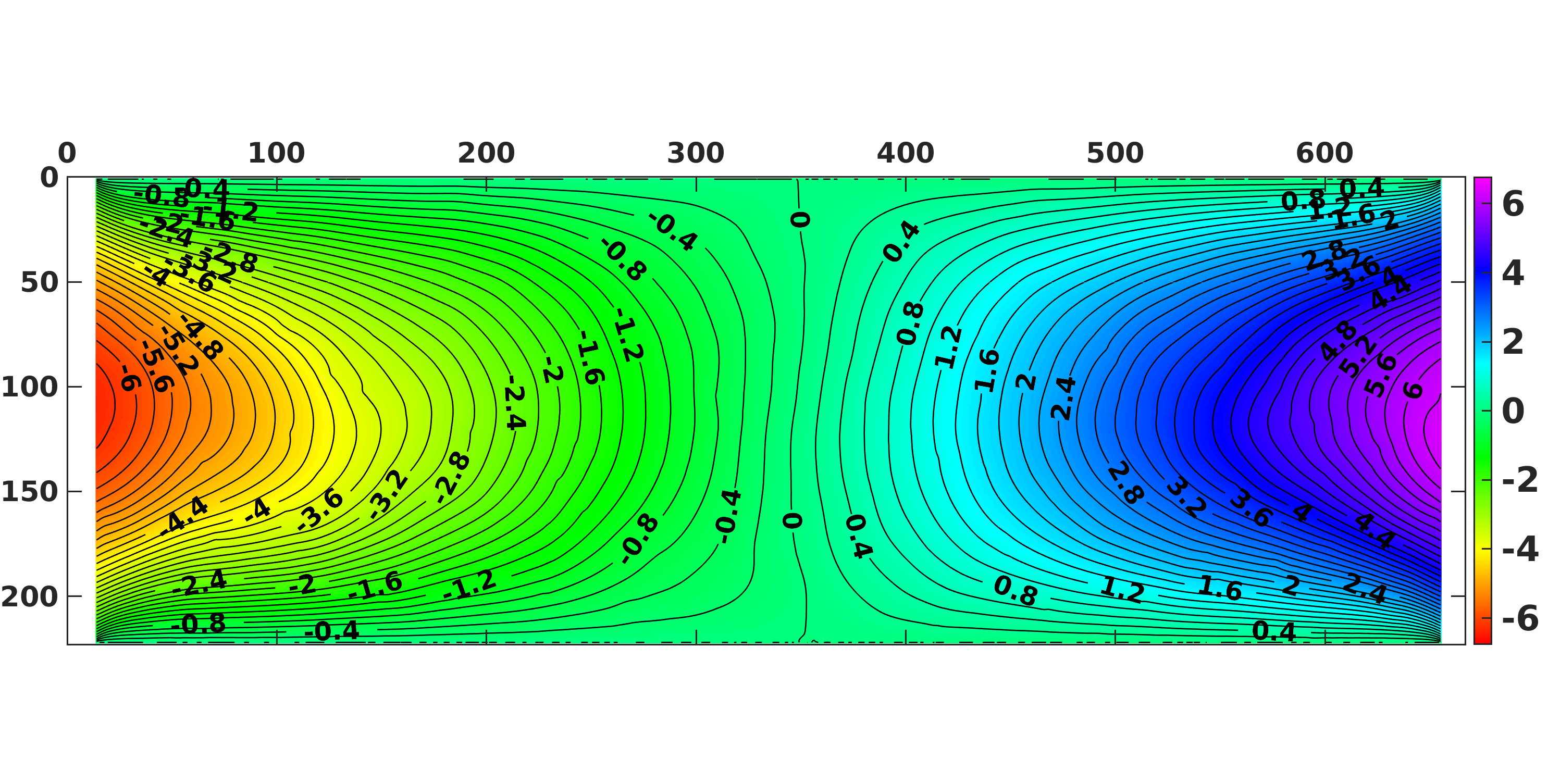}
    \\
    \includegraphics[width=0.49\textwidth, clip=true, trim={0cm 4cm 0cm 3cm}]{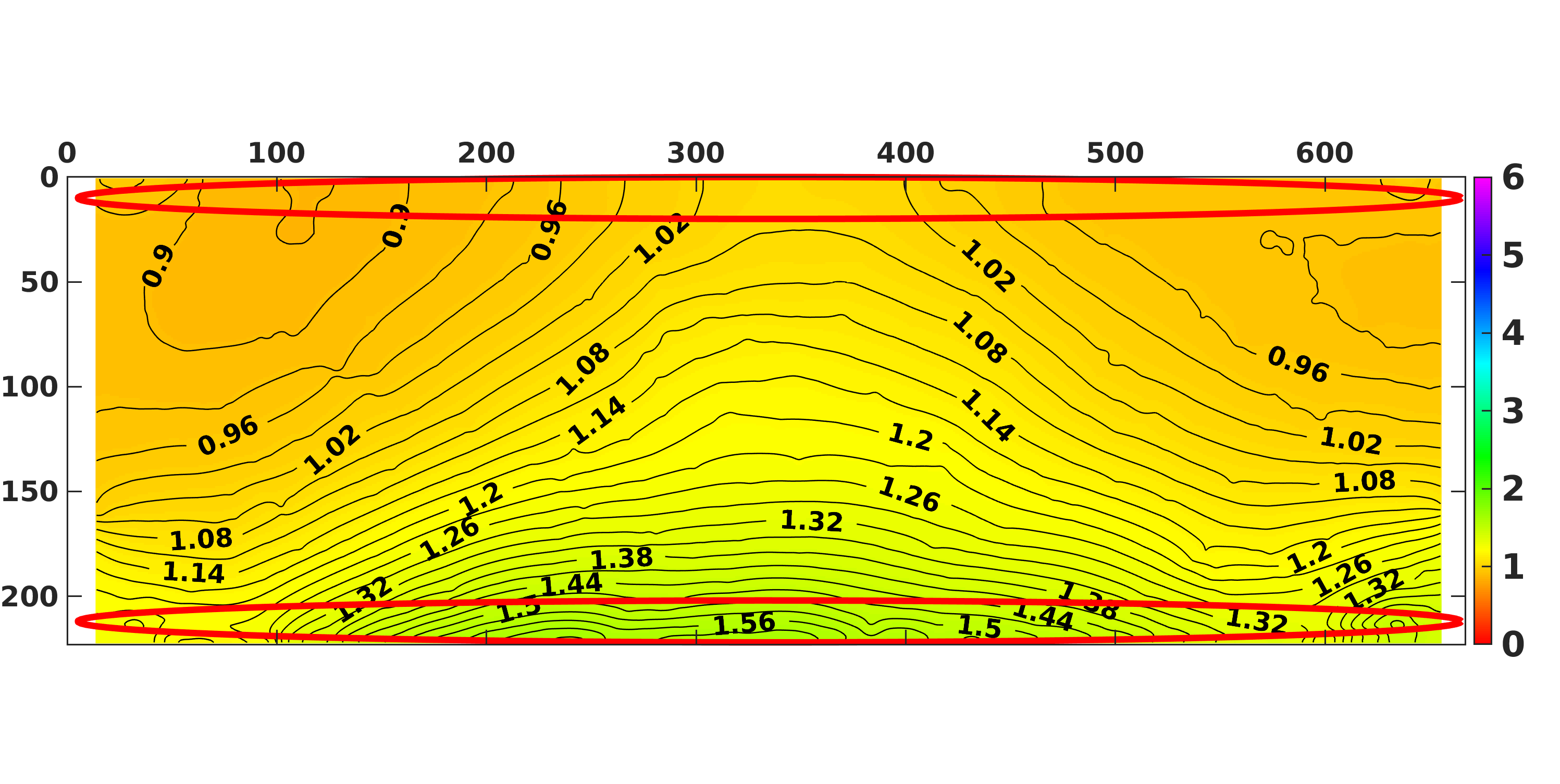}
    \includegraphics[width=0.49\textwidth, clip=true, trim={0cm 4cm 0cm 3cm}]{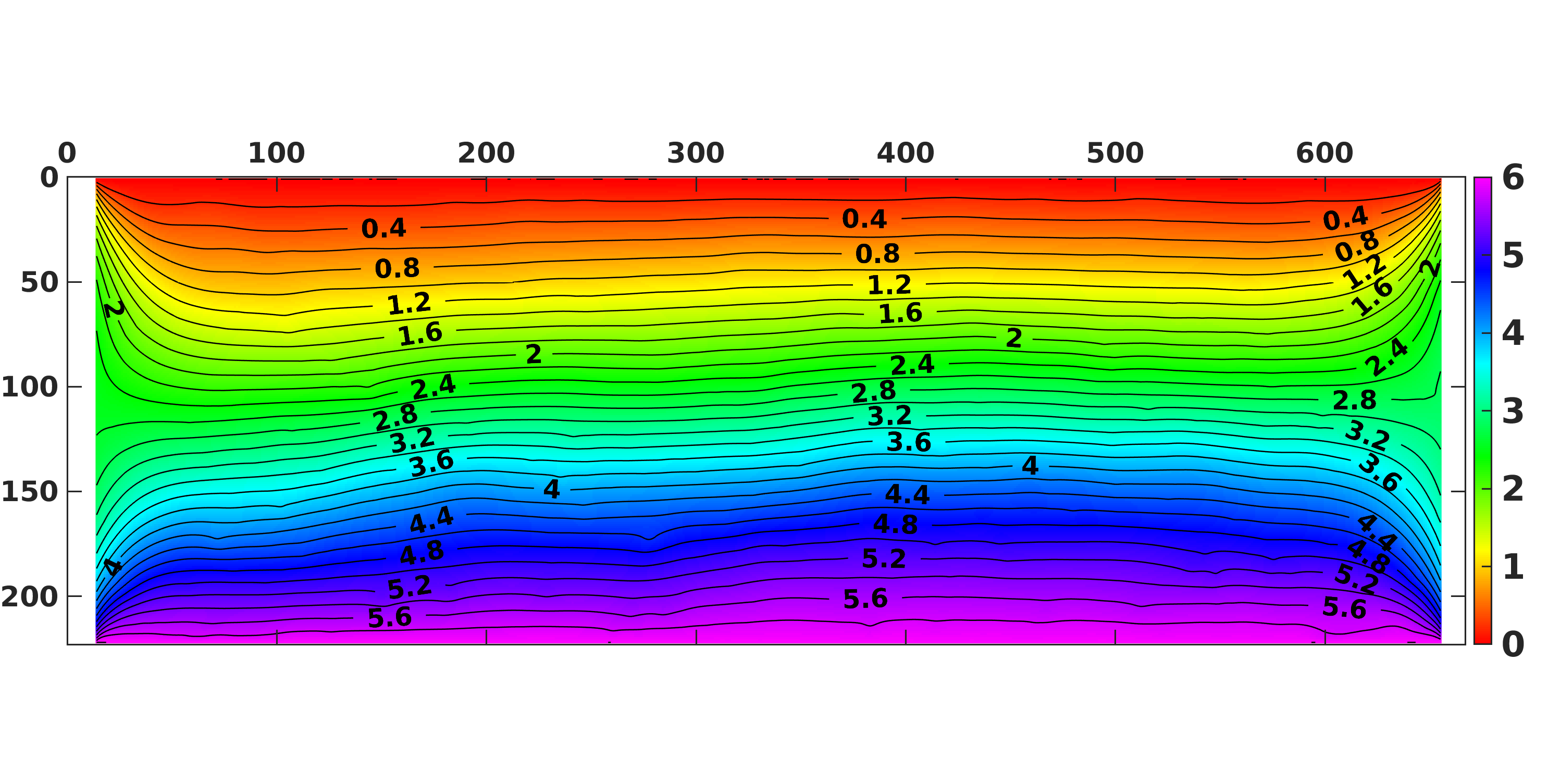}
    \caption{Lateral (top) and axial (bottom) components of the displacement fields estimated by standard optical flow (left) and our proposed elastographic optical flow method (right) applied to the tomograms depicted in Figure~\ref{fig_exp_oct}. Circled in red: level lines which approach the boundary perpendicularly due to natural boundary conditions implicit in standard optical flow. This does not agree with the motion occurring in the elastography experiment.}
    \label{fig_results}
\end{figure}

In this paper, we focus on the displacement field estimation in two-step approaches to elastography; see e.g.~\cite{Hubmer_Sherina_Neubauer_Scherzer_2018,Schmitt_1998,Wang_Larin_2015,Wijesinghe_Kennedy_Sampson_2020} and the references therein. In particular for quantitative results in elastography, accurately estimating this internal displacement field is crucial for obtaining reliable material parameter estimates. However, up to now most of the research on two-step methods for elastography has been concerned with the second step, i.e., the parameter estimation from given displacement fields. Apart from some exceptions discussed below, the first step, i.e., the displacement estimation itself, has been implement with standard optical flow methods. As a result, it can be observed that while many of the proposed methods for material parameter reconstruction work well on simulated displacement fields, their performance with experimental data is limited. In the course of working with different data, see e.g.\ Figure~\ref{fig_exp_oct}, we became convinced that this is due to the fact that important physical constraints are not adequately accounted for by standard optical flow techniques: These are a lack of physical assumptions on the motion, an improper treatment of boundary conditions, and peculiarities of the data and the post-processing in scattering-based imaging modalities. From the physical point of view, the motion observed in an experiment with an elastic sample changing its shape due to external forces, see e.g.~Figure~\ref{fig_exp_oct}, is a non-rigid body transformation, where the distance between the neighboring points changes non-linearly. Secondly, the visibility of a sample's internal structure and its geometry affects the motion estimation quality. The standard formulation of optical flow relies on the assumption of the similarity of sequential images, i.e., of their brightness remaining constant. In scattering-based imaging techniques, tomograms are formed by particles in the material which reflect the electromagnetic waves. This means that optical flow can detect motion in a semi-transparent uniform material only in the presence of reflectors inside. The material needs to contain particles or needs to be artificially seeded with reflectors in order to make the deformation accurately visible. Experimental data frequently happens to violate the brightness constancy and physicality in one or another way: the refractive properties of the material change under compression; imaging artifacts; borders of the sample perpendicular to the imaging direction are either invisible or they are removed from the data during post-processing by cutting out to the region of interest inside. As a result, optical flow is ``too slow'' to follow a depth-varying rate of motion in uniform materials with sparse reflectors, and underestimates the flow near the samples borders by applying the built-in natural boundary conditions which are physically not correct. This can be seen in the level lines of the flow reaching the boundary perpendicularly, see e.g.~Figure~\ref{fig_results} (top left, encircled in red). 

In this paper, we propose an elastographic optical flow (EOFM) method which takes into account additional experimental and physical side constraints. 
We concentrate on EOFM based on quasi-static imaging from (a pair of) successive images. Similar techniques can also be used for quasi-static, harmonic, and  transient imaging. In particular we consider the proper treatment of boundary conditions, the use of speckle information, as well as structure information derived from knowledge of the background material. Based on both simulated and real experimental data we shall see that by combining the considered techniques we can obtain physically meaningful displacement field.

% % % % % % % % % % % % % % % % % % % % % %
% Section - Displacement Field Estimation %
% % % % % % % % % % % % % % % % % % % % % %
\section{Displacement Field Estimation via Optical Flow}\label{sect_opt_flow}

The basis for displacement field estimation typically is the  optical flow equation
    \begin{equation} 
		\nabla I \cdot \uf + I_t =0 \,.
	\label{eq_opt_flow}
	\end{equation}
It connects an image intensity function $I = I(\x,t)$ with a displacement (motion, flow) field $\uf(\x) = (u_1(\x),u_2(\x))^T$ for $\x \in \Omega \subset \R^2$. Based on this equation, an estimate of the displacement field $\uf$ is typically found by minimizing
	\begin{equation} 
		J(\uf) := \intVx{\left(\nabla I \cdot \uf + I_t\right)^2} +  \alpha \RE(\uf)\,,
		\qquad
		\forall \, t > 0 \,,
	\label{def_J}
	\end{equation}
where $\alpha \geq 0$ is a regularization parameter and $\RE$ is some suitably defined regularization functional. In this paper, we focus on the common choice
	\begin{equation*}
	\begin{split}
		\RE(\uf) &:= \norm{\nabla \uf}^2_{\LtO} := \norm{\nabla u_1}_\LtO^2 + \norm{\nabla u_2}_\LtO^2\,,
	\end{split}
	\end{equation*}
which gives rise to the well-known Horn-Schunck method. It enforces certain smoothness constraints on the displacement field $\uf$ and can also be given a physical interpretation \cite{Schnoerr_1991}. Over the years, a wide variety of advanced motion estimation techniques have been proposed; see e.g.~\cite{Aubert_Kornprobst_2006,Baker_Scharstein_Lewis_Roth_Black_Szeliski_2011,Black_Anandan_1996,Brox_Malik_2011,Chen_Jin_Lin_Cohen_Wu_2013,Sun_Roth_Black_2013,Weickert_Bruhn_Brox_Papenberg_2006} and the references therein. Since at their core most of them still follow a similar strategy, we use \eqref{def_J} as our starting point for all further considerations.

% Subsection - Speckle Tracking 
\subsection{Speckle Tracking}\label{sect_speck_tr}

The phenomenon known as speckle is commonly observed in scattering imaging modalities such as ultrasound, optical coherence tomography (OCT), radar-based imaging, and radio-astronomy. Resulting from the constructive and destructive interference of back-scattered waves, it is responsible for granulated images. These speckle patterns are influenced by a number of factors such as the design of the imaging system or optical properties of the imaged samples.

On the one hand, speckle can be seen as a source of noise corrupting the obtained images. On the other hand, they also contain important information, such as on the motion inside the sample; see e.g.~\cite{Schmitt_Xiang_Yung_1998}. This even lead to the introduction of additional \cite{Schmitt_1998} or virtual \cite{Glaatz_Scherzer_Widlak_2015,Schmid_Zabihian_Widlak_Glatz_Liu_Drexler_Scherzer_2015} speckle. Since speckle appears throughout the obtained images, and thus in particular also in otherwise featureless areas, their movement during an elastography experiment can also be used to estimate the internal displacement field more accurately. A popular method for doing so is the normalized cross-correlation method \cite{Duncan_Kirkpatrick_2001,Schmitt_1998}. Unfortunately, if the size of the correlation area is not carefully chosen, or if the applied strain is too small or too high, then this method is prone to miss-estimations. In addition, its pixel-by-pixel processing is very time-consuming.

Hence, in \cite{Sherina_Krainz_Hubmer_Drexler_Scherzer_2020} the authors proposed a novel, heuristics-based image-processing algorithm for the detection and tracking of large speckle formations, termed \emph{bubbles}, which is also used here.

Now, assume that from a given pair of successive images the centers of mass $\xhi =(\hat{x}^i_1,\hat{x}^i_2) \in \Omega$ and the directions of motion $\ufhi =(\hat{u}^i_1, \hat{u}^i_2) \in \R^2$ of a number $M$ of large speckle formations (bubbles) have been extracted. In \cite{Sherina_Krainz_Hubmer_Drexler_Scherzer_2020}, the authors proposed to complement the optical flow functional $J(u)$ by the functional 
	\begin{equation}
		\S_\sigma(\uf) := \sum\limits_{i=1}^{M}\intVx{g_\sigma(\x,\hat{\x}^i) \abs{\uf(\x)-\hat{\uf}^i}^2} \,, 
	\label{def_S_u}
	\end{equation}
where the Gaussian-functions $g_\sigma(\x,\hat{\x}^i)$ are defined by
	\begin{equation*}
		g_\sigma(\x,\hat{\x}^i) = \frac{1}{2\pi \sigma^2} e^{-\frac{(x_1-\hat{x}_1^i)^2 + (x_2-\hat{x}_2^i)^2}{2\sigma^2}} \,.
	\end{equation*}
This leads to the optical flow method, consisting in minimization of
	\begin{equation}
	\begin{split}
		J(\uf) + \beta \S_\sigma(\uf) = \intVx{\left(\nabla I \cdot \uf + I_t\right)^2} +  \alpha \RE(\uf) + \beta \S_\sigma(\uf) \,,
	\end{split}
	\label{def_J_bS}
	\end{equation}
where $\beta \geq 0$ is another regularization parameter. Depending on the application, the quality of the data, and any given a-priory assumptions on the displacement field $\uf$, the values of $\alpha$, $\beta$, and $\sigma$ can be adjusted to put an emphasis either on the smoothness of the field, or the fit to the given bubble motion $\ufhi$.

% Subsection - Boundary Conditions 
\subsection{Boundary Conditions}\label{sect_bd_cond}

Consider the minimization of the Horn-Schunck functional $J(\uf)$ defined in \eqref{eq_opt_flow} for a fixed time $t$. Without any additional restrictions, its minimizer $\uf$ can be seen to satisfy the so-called natural boundary conditions
    \begin{equation}
        \nabla u_1 \cdot \nv = 0 
        \qquad
        \text{and}
        \qquad
        \nabla u_2 \cdot \nv = 0 \,,
    \label{bdc_natural}
    \end{equation}
where $\nv$ denotes a unit vector perpendicular to the boundary $\partial \Omega$. The same boundary conditions also hold for the minimizer of the adapted functional defined in \eqref{def_J_bS}, i.e., when speckle information is included in the reconstruction.

However, in most cases these natural boundary conditions do not agree with the physical boundary conditions imposed by an actual elastography setup. Consider e.g.\ a sample which is fixed to a stable surface along a part $\Gamma_1 \subset \partial \Omega$ of its boundary $\partial \Omega$, and which is compressed by an amount $g$ at another part $\Gamma_2$ of its boundary. This can be expressed by the Dirichlet boundary conditions 
    \begin{equation}
        \uf = 0  \quad \text{on } \Gamma_1 \,,
        \qquad
        \text{and}
        \qquad
        \uf = g  \quad \text{on } \Gamma_2 \,, 
    \label{bdc_Dirichlet}
    \end{equation}
which clearly differ from the natural boundary conditions \eqref{bdc_natural}. Hence, in this situation the standard Horn-Schunck optical flow algorithm would yield an estimate of the internal displacement field which is not physically meaningful.

Two possibilities for incorporating known Dirichlet boundary conditions of the form \eqref{bdc_Dirichlet} into the displacement field reconstruction suggest themselves. The first is to introduce an additional penalty term of the form
    \begin{equation}
        \mathcal{B}(\uf) := \int_{\Gamma_1 \cup \Gamma_2}  \abs{\uf - g}^2  \,dS 
    \label{def_B}
    \end{equation}
which can be used to enforce the Dirichlet boundary conditions in a weak form. The other is to restrict the search space in the minimization of either \eqref{def_J} or \eqref{def_J_bS} to contain only those functions which satisfy the required boundary conditions.

Another situation in which the question of proper boundary conditions becomes particularly relevant is when one can only work with measurements of a certain region within the sample. This is e.g.\ the case if parts of the measurements are corrupted by strong noise or artefacts and thus have to be removed, or if only this region was imaged to begin with. It should be clear that in this case the natural boundary conditions \eqref{bdc_natural} are even less appropriate than before.

In order to deal with this issue, we propose the following strategy: For those boundaries of the measurement region which overlap with the sample boundary, available boundary conditions like \eqref{bdc_Dirichlet} can be used as described above. For the remaining boundaries, we propose to use the motion information contained in bubbles to obtain physically meaningful displacement fields. Typically, this can be accomplished by adding the penalty term $\S_\sigma(\uf)$ defined in \eqref{def_S_u} and a proper tuning of the corresponding parameters $\beta$ and $\sigma$; compare with \eqref{def_J_bS}. In case that only comparatively few bubbles are located near the boundaries without available conditions, it can be advantageous to replace $\S_\sigma(\uf)$ by the functional
    \begin{equation*}
        \tilde{\S}_\sigma(\uf) := \sum\limits_{i=1}^{M} \beta_i \intVx{g_\sigma(\x,\hat{\x}^i) \abs{\uf(\x)-\hat{\uf}^i}^2} \,,
    \end{equation*}
and to emphasize the motion information contained in bubbles close to the boundaries by adapting the corresponding values of the parameters $\beta_i$.

% Subsection - Homogeneous Background Information
\subsection{Homogeneous Background Information}\label{sect_hom_bg}

General knowledge on the expected structure of the sought for displacement field can be used to enhance the overall quality of the reconstruction methods. Consider e.g.\ the case that a sample consists of multiple different inclusions in an otherwise homogeneous background material. Then one can write
	\begin{equation}
		\uf = \ufbg + \ufupd \,,
	\label{Ansatz}
	\end{equation}
where $\ufbg$ denotes the displacement field which would result from the same elastography experiment carried out on the same sample but without any inclusions, and $\ufupd$ denotes an update which amends this field. Assuming e.g. that the material parameters of the background material are known, then the field $\ufbg$ can be computed by applying a suitable forward model. Hence, the task of estimating the displacement field $\uf$ reduces to finding the update field $\ufupd$. This can be done by adapting the reconstruction approach outlined above as follows: Since $\uf = \ufbg + \ufupd$ should satisfy the optical flow equation \eqref{eq_opt_flow}, we can adapt \eqref{def_J_bS} and determine the update field $\ufupd$ as the minimizer of the functional
	\begin{equation*}
	\begin{split}
		\intVx{ \kl{\nabla I \cdot\kl{ \ufbg + \ufupd}+ I_t }^2 }  +  \alpha \RE(\ufupd) + \beta \Sb_\sigma(\ufupd) \,,
	\end{split}
	\end{equation*}	
where now the speckle functional $\S_\sigma$ defined in \eqref{def_S_u} is replaced by the functional
    \begin{equation*}
		\Sb_\sigma(\ufupd) :=
		\sum\limits_{i=1}^{M}\intVx{g_\sigma(\x,\hat{\x}^i) \abs{\ufupd(\x)-\left(\hat{\uf}^i - \ufbg(\hat{\x}^i)\right)}^2} \,, 
	\end{equation*}	
which accounts for the relative shifts induced by \eqref{Ansatz}. This also effects the question of appropriate boundary conditions. For example, in the case of the Dirichlet boundary conditions \eqref{bdc_Dirichlet}, these have to be satisfied for both $\uf$ and for $\ufbg$, and thus $\ufupd$ has to satisfy
    \begin{equation*}
        \ufupd = 0  \quad \text{on } \Gamma_1 \cup \Gamma_2 \,.
    \end{equation*}
Furthermore, if one assumes that there holds $\uf \approx \ufbg$ on $\partial \Omega$, then homogeneous Dirichlet boundary conditions $\ufupd = 0$ can also be applied on those parts of the boundary where no other physically motivated boundary conditions can be used (cf. Section~\ref{sect_bd_cond}). Even though this may only be a rough approximation, it nevertheless results in a more meaningful condition than the natural boundary condition \eqref{bdc_natural}, and together with the speckle information typically helps to improve the overall reconstruction quality.

% % % % % % % % % % % % % % % % % % %
% Section - Reconstruction Approach %
% % % % % % % % % % % % % % % % % % %
\section{The Elastographic Optical Flow Method}\label{sect_recon}

In this section, we formulate the elastographic optical flow method (EOFM) for the problem of determining the internal displacement field of a sample which is subjected to a deformation of the form \eqref{bdc_Dirichlet} and imaged with some scattering imaging modality. It takes into account prior knowledge on the homogeneous background material and tracking of bubbles. For this, let $\uf = \ufbg + \ufupd$ denote the decomposition of $\uf$ described in Section~\ref{sect_hom_bg}, and the field $\ufbg$ be known. Then EOFM determines $\ufupd$ as the minimizer of
    \begin{equation}
	\begin{split}
		F(\vf) := \intVx{ \kl{\nabla I \cdot\kl{ \ufbg + \vf}+ I_t }^2 }  +  \alpha \RE(\vf) + \beta \Sb_\sigma(\vf)  \,,
	\end{split}
	\label{def_F}
	\end{equation}
where the space over which this functional is minimized is 
    \begin{equation*}
        V :=  \Kl{ \vf \in \HoO^2 \, \vert \, \vf = 0 \text{ on } \Gamma_1 \cup \Gamma_2} \,.
    \end{equation*}

In order to analyse this approach and compute a minimizer of \eqref{def_F} we adopt the ideas of Schn\"orr \cite{Schnoerr_1991}. That is, we rewrite $F$ in the form
	\begin{equation}
		F(\vf) = \frac{1}{2}a(\vf,\vf) - b(\vf) + c \,,
	\label{F_a_b_c}
	\end{equation}
where the bilinear form $a(\cdot,\cdot)$ and the linear form $b(\cdot)$ are given by
    \begin{equation*}
    \begin{split}
        & a(\uf,\vf) := 
        2 \intVx{ \kl{ \left( \nabla I \cdot \uf \right)\left( \nabla I \cdot \vf \right) +  \alpha \nabla \uf:\nabla \vf  +  \beta \sum\limits_{i=1}^{M} g_\sigma(\x,\hat{\x}^i) \left(\uf \cdot \vf\right)}}  \,,
        \\
        & b(\vf) :=
        2 \intVx{\kl{\beta \sum\limits_{i=1}^{M}  g_\sigma(\x,\hat{\x}^i) (\hat{\uf}^i - \ufbg(\hat{\x}^i)) \cdot \vf -\kl{I_t +  \nabla I \cdot \ufbg } \left( \nabla I \cdot \vf \right) }} \,,
	\end{split}
	\label{def_a_b_c}
	\end{equation*}
and $c$ denotes a constant term. Using this representation \eqref{F_a_b_c}, we obtain 
 
\begin{theorem}\label{thm_main}
Let $\Omega \subset \R^2$ be a nonempty, bounded, open, and connected set with a Lipschitz continuous boundary $\partial \Omega$ and let $\grad I \in \LiO$ and $I_t \in \LtO$. Furthermore, let $\alpha > 0$ and let the components of $\grad I$ be linearly independent. Then, the unique minimizer of the problem
 	\begin{equation}\label{min_F}
 		\min\limits_{\uf \in V} F(\uf) \,, 
 	\end{equation}
is given as the unique solution $\uf \in V$ of the linear problem
	\begin{equation}\label{a_b}
		a(\uf,\vf) = b(\vf) \,, \qquad \forall \, \vf \in V \,.
	\end{equation}
Furthermore, the solution of this equation depends continuously on the right-hand side $b(\cdot)$, but not necessarily on the image intensity function $I$.
\end{theorem}
\begin{proof}
This follows from the representation \eqref{F_a_b_c} in the same way as in \cite{Schnoerr_1991,Sherina_Krainz_Hubmer_Drexler_Scherzer_2020}.
\end{proof}

% % % % % % % % % % % % % % % %
% Section - Numerical Results %
% % % % % % % % % % % % % % % %
\section{Numerical Results}\label{sect_numerics}

In this section, we present some numerical results of the application of our proposed EOFM approach based on both simulated and experimental data. The implementation and computational environment is adapted from \cite{Sherina_Krainz_Hubmer_Drexler_Scherzer_2020}. We also combine EOFM with a standard multi-scale strategy \cite{Lauze_Kornprobst_Memin_2004,MeinhardtLlopis_SanchezPerez_Kondermann_2013,Modersitzki_2009}, which by itself is not sufficient for satisfactory results, but is seen to provide additional accuracy to our proposed approach.

% Subsection - Simulated Data
\subsection{Simulated Data}

First, we consider a synthetic sample consisting of a circular inclusion within a homogeneous background of different stiffness, a simulated scattering image, which is shown in Figure~\ref{fig_sam} (left) with $200$ randomly distributed bubbles. This sample is assumed to be fixed on top and compressed from the bottom, which corresponds to the Dirichlet boundary conditions \eqref{bdc_Dirichlet}, and allowed to move freely on the sides. Under the model of linearized elasticity (see \cite{Hubmer_Sherina_Neubauer_Scherzer_2018} or \cite{Sherina_Krainz_Hubmer_Drexler_Scherzer_2020}), this results in a displacement field $\uf$ with components $u_1$, $u_2$ depicted in Figure~\ref{fig_sim_f}. The resulting compressed sample is given in Figure~\ref{fig_sam} (right). For extracting the motion $\ufhi$ of added bubbles for \eqref{def_S_u}, we used speckle tracking from Section~\ref{sect_speck_tr}. 

\begin{figure}[t!]
    \centering
	\includegraphics[width=0.48\textwidth, clip=true, trim={0cm 4cm 0cm 3cm}]{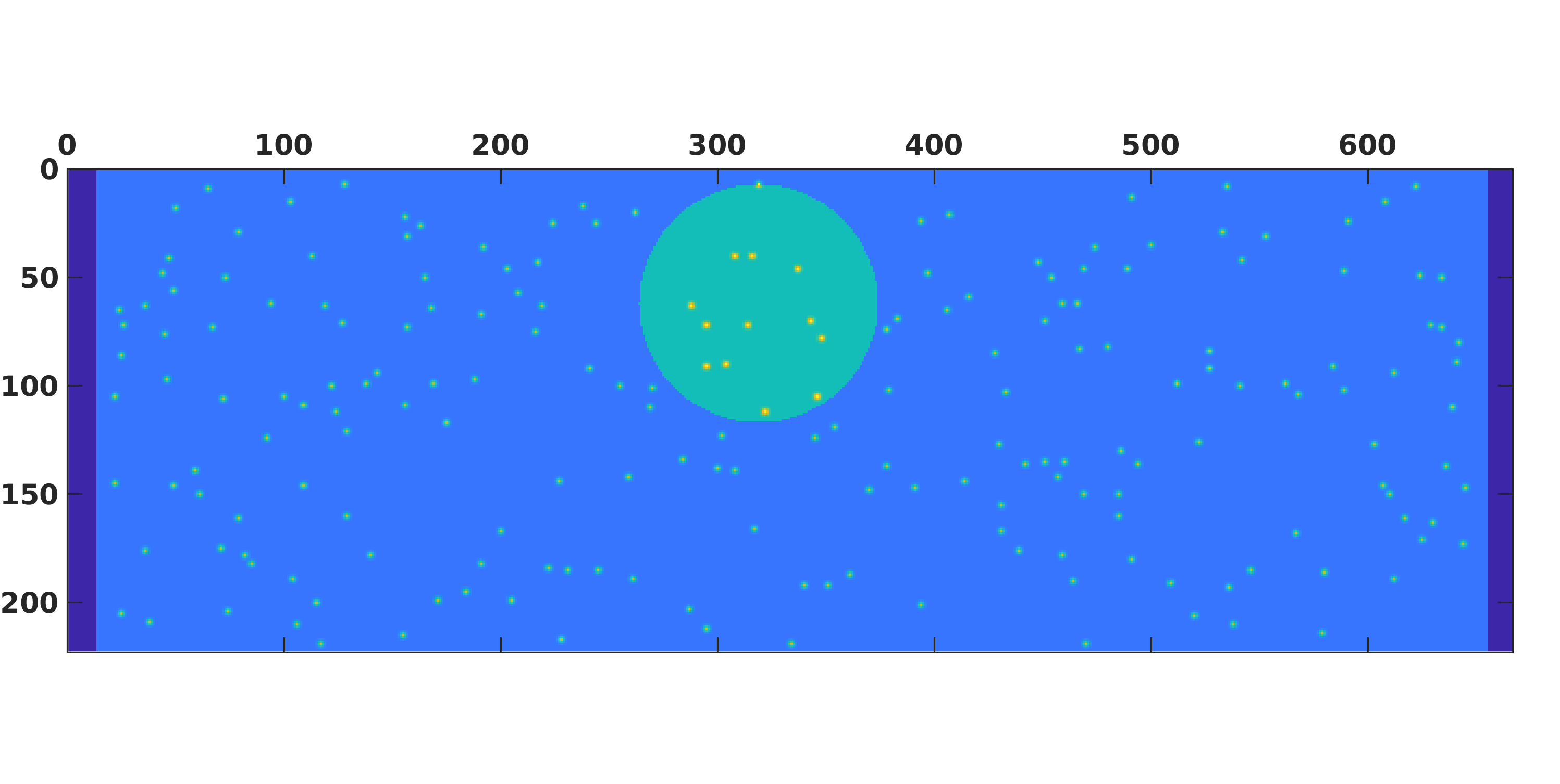}
	\includegraphics[width=0.48\textwidth, clip=true, trim={0cm 4cm 0cm 3cm}]{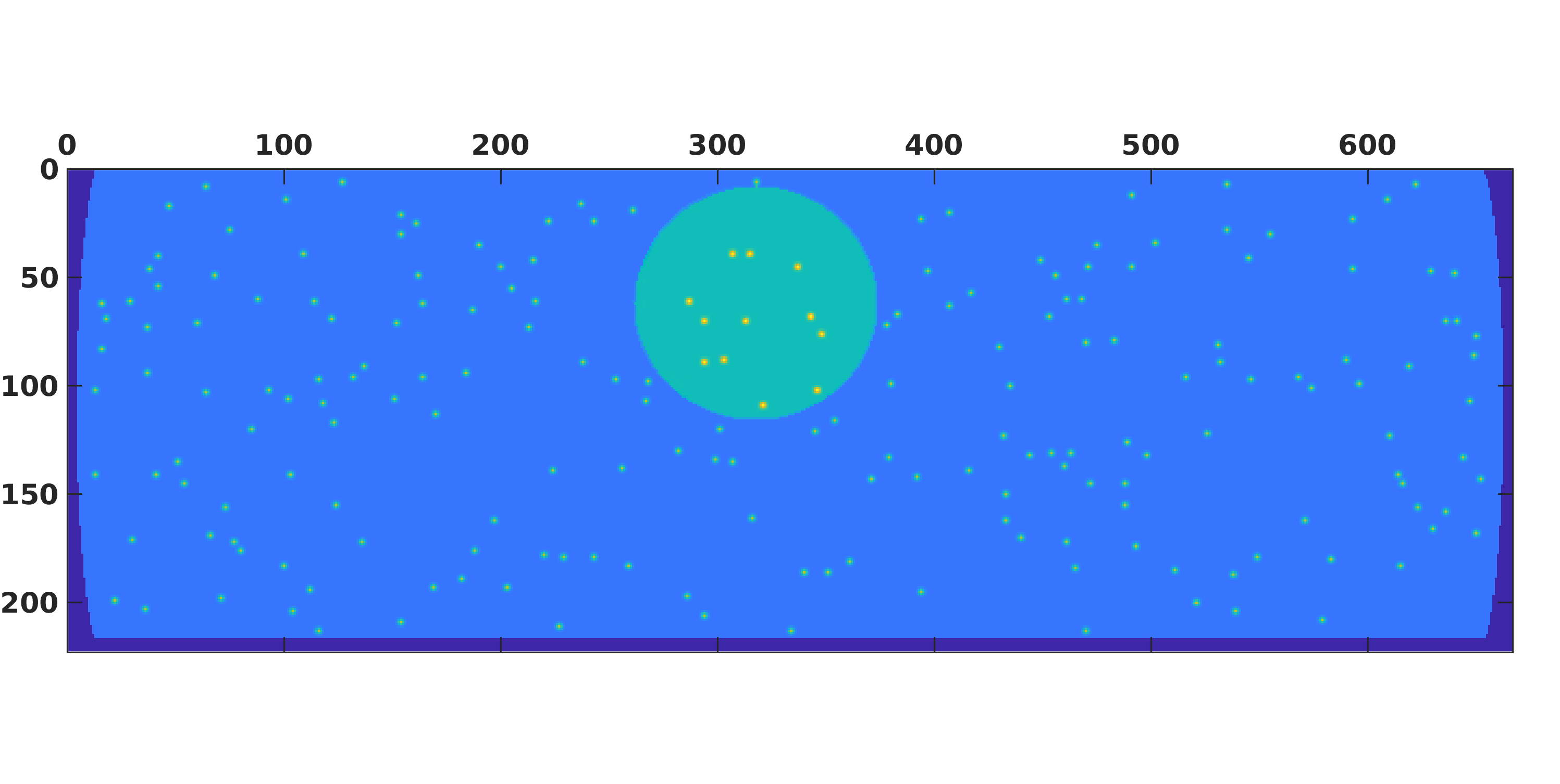}
	\caption{Simulated sample with randomly distributed speckle formations (bubbles) before (left) and after (right) compression.}
	\label{fig_sam}
\end{figure}

\begin{figure}[t!]
    \includegraphics[width=0.49\textwidth, clip=true, trim={0cm 4cm 0cm 3cm}]{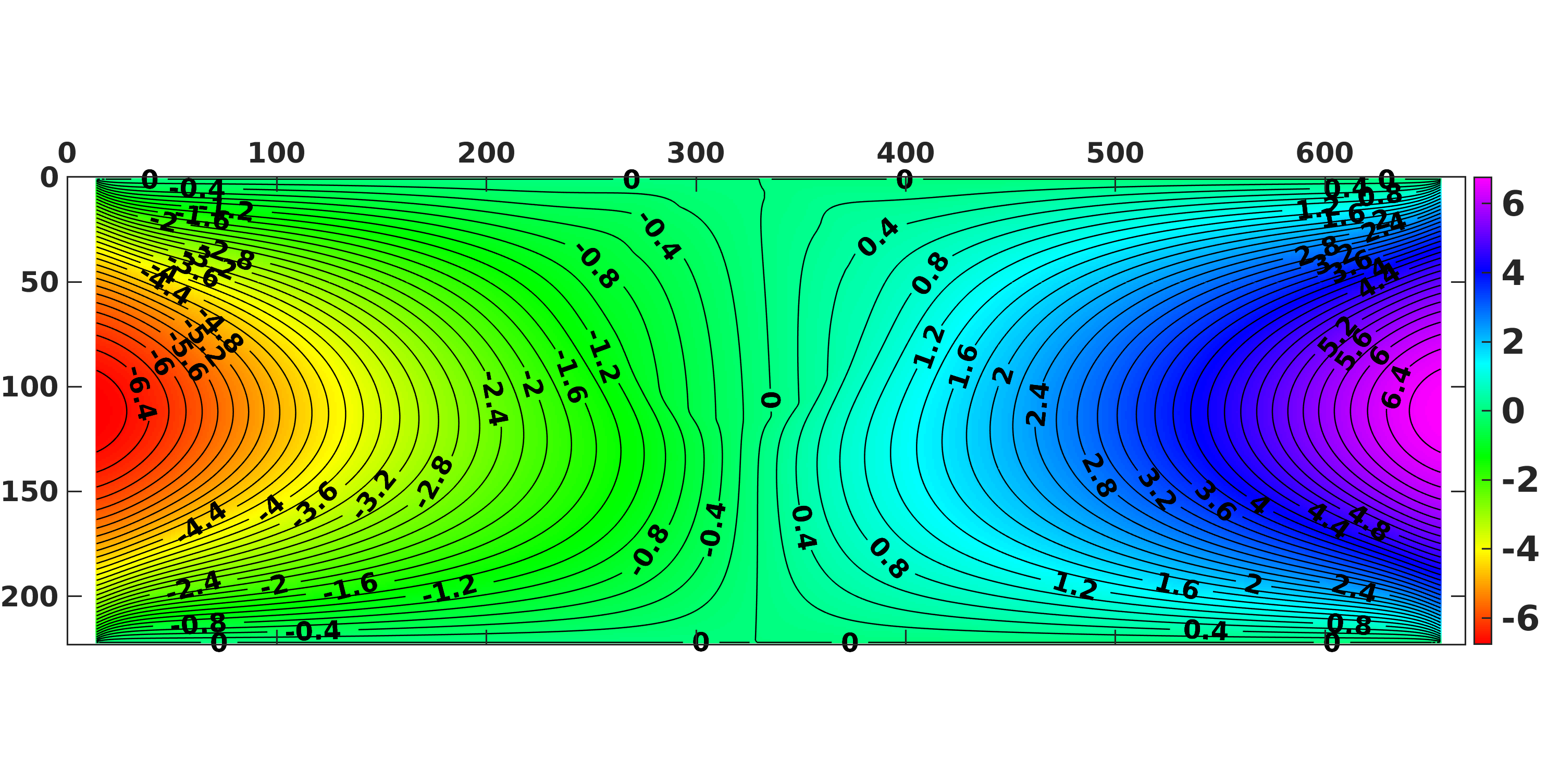}
    \includegraphics[width=0.49\textwidth, clip=true, trim={0cm 4cm 0cm 3cm}]{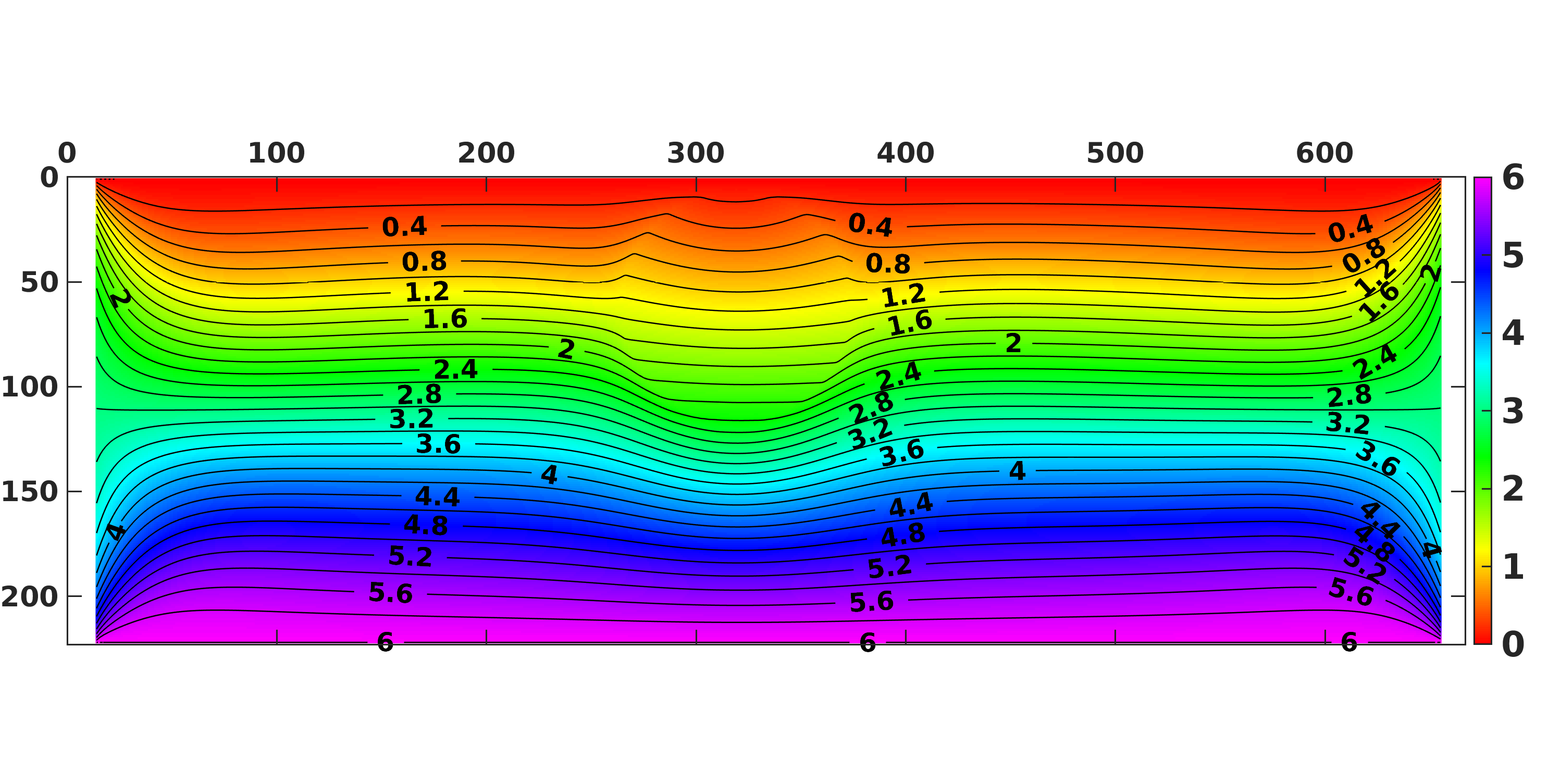}
    \caption{Components $u_1$ (left) and $u_2$ (right) of the simulated displacement field $\uf$.}
    \label{fig_sim_f}
\end{figure}

In order to illustrate the effects of different optical flow methods explained in Sections~\ref{sect_speck_tr}, \ref{sect_bd_cond}, \ref{sect_hom_bg} on the reconstruction of the displacement field, we run tests with various combinations of regularization terms and side constraints. Based on the independent parameter study, we chose $\alpha=0.8$ for smoothness-regularization and $\beta=0.5$, $\sigma=5$ when utilizing the speckle-regularization. For using the multi-scale approach, we took $4$ scale-levels (see  \cite{Sherina_Krainz_Hubmer_Drexler_Scherzer_2020} for details). In the first test, we solve the minimisation problem for \eqref{def_J_bS} with the bubble motion information. Figure~\ref{fig_sim_exp_0} depicts the results of the displacement estimation and absolute error in the field components. The speckle-regularization improves the flow estimate in $u_2$, with its absolute error being up to $0.6$ pixel in the lower half of the sample. However, the lateral motion is underestimated by up to $2.6$ pixel in the border area. Next, we minimise \eqref{def_F} with the background field information induced by the same boundary displacement \eqref{bdc_Dirichlet} known from the experiment. The relative errors for the resulting estimates are collected in Table~\ref{tab_rel_err}. Figures~\ref{fig_sim_exp_2} and  \ref{fig_sim_exp_4} depict the fields reconstructed without ($\beta=0$) and with ($\beta=0.5$) the speckle-regularization in addition to the smoothness-regularization and the multi-scale approach. The latter result shows only the maximal error of $0.35$ pixel in $u_2$ and $0.6$ pixel in $u_1$. 

\begin{figure}[t!]
    \includegraphics[width=0.49\textwidth, clip=true, trim={0cm 4cm 0cm 3cm}]{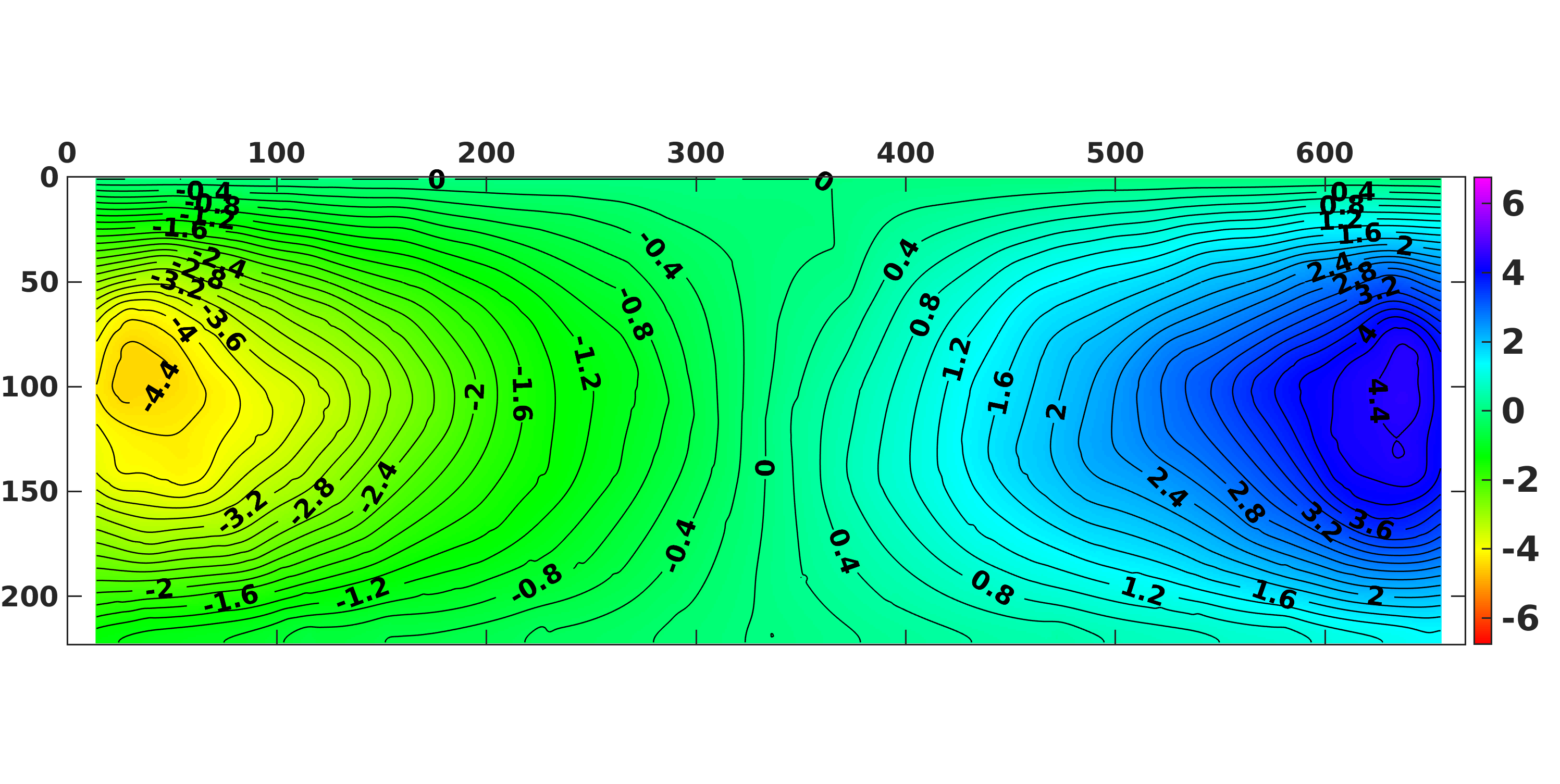}
    \includegraphics[width=0.49\textwidth, clip=true, trim={0cm 4cm 0cm 3cm}]{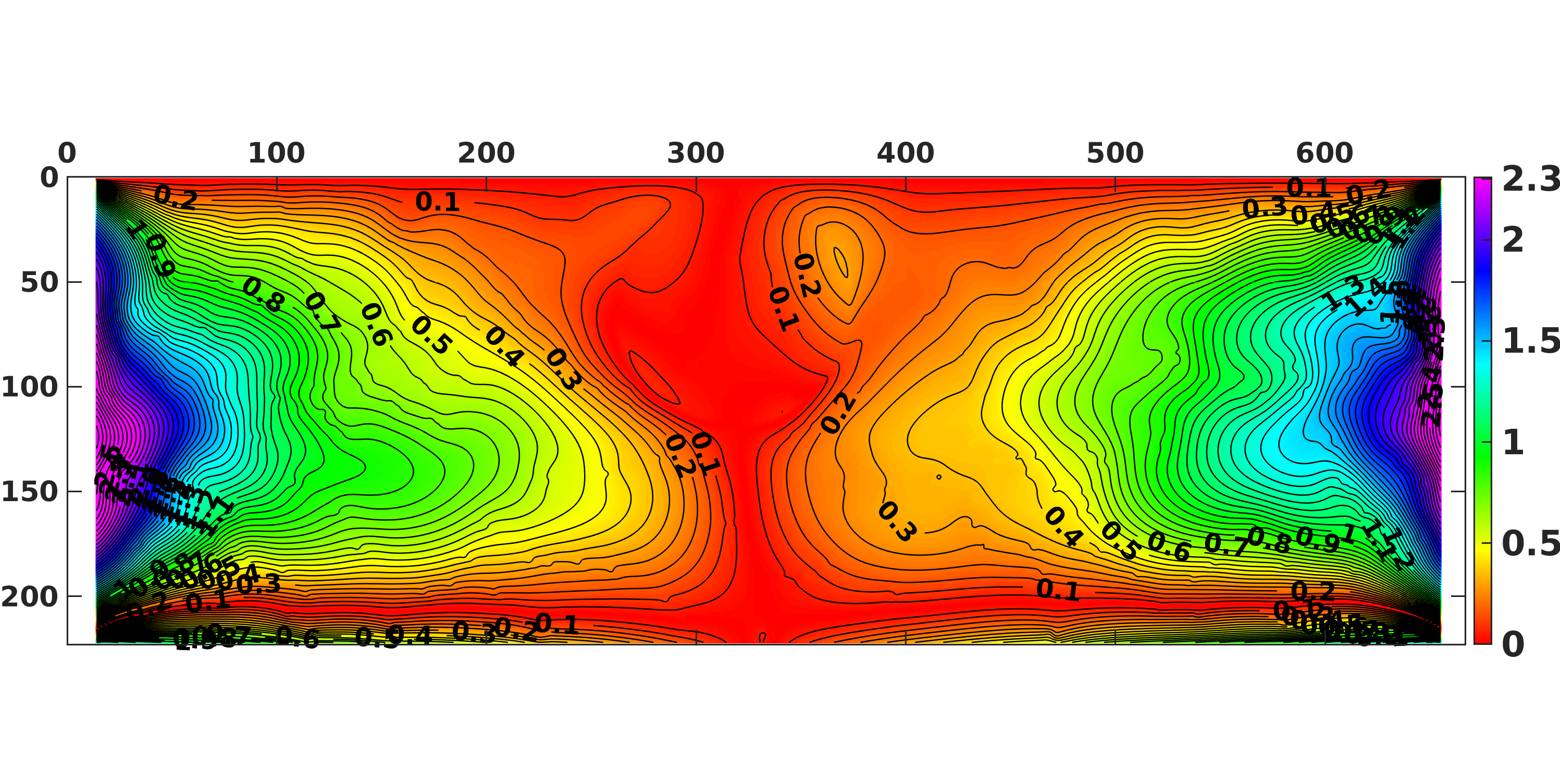}
    \\
    \includegraphics[width=0.49\textwidth, clip=true, trim={0cm 4cm 0cm 3cm}]{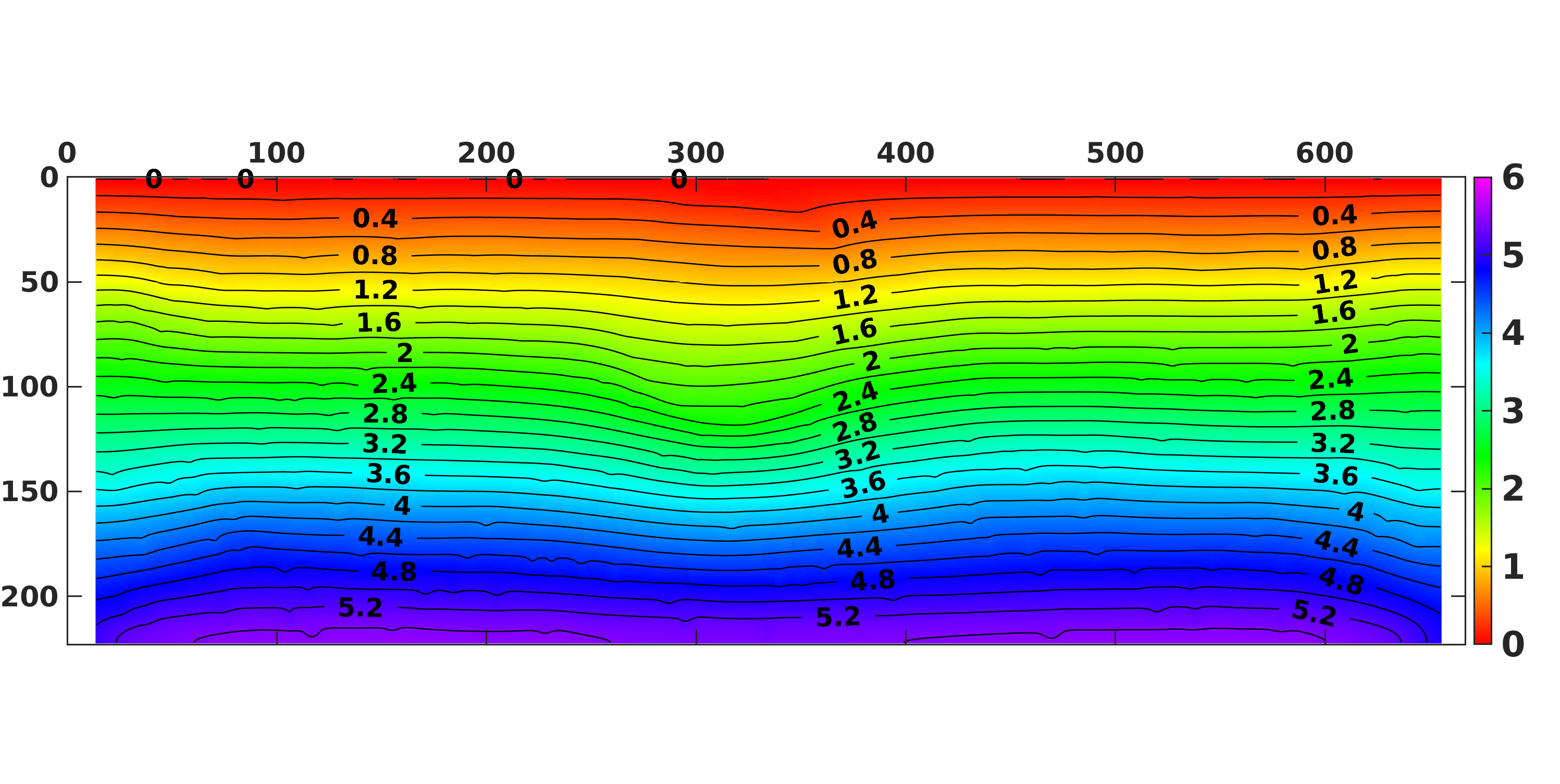}
    \includegraphics[width=0.49\textwidth, clip=true, trim={0cm 4cm 0cm 3cm}]{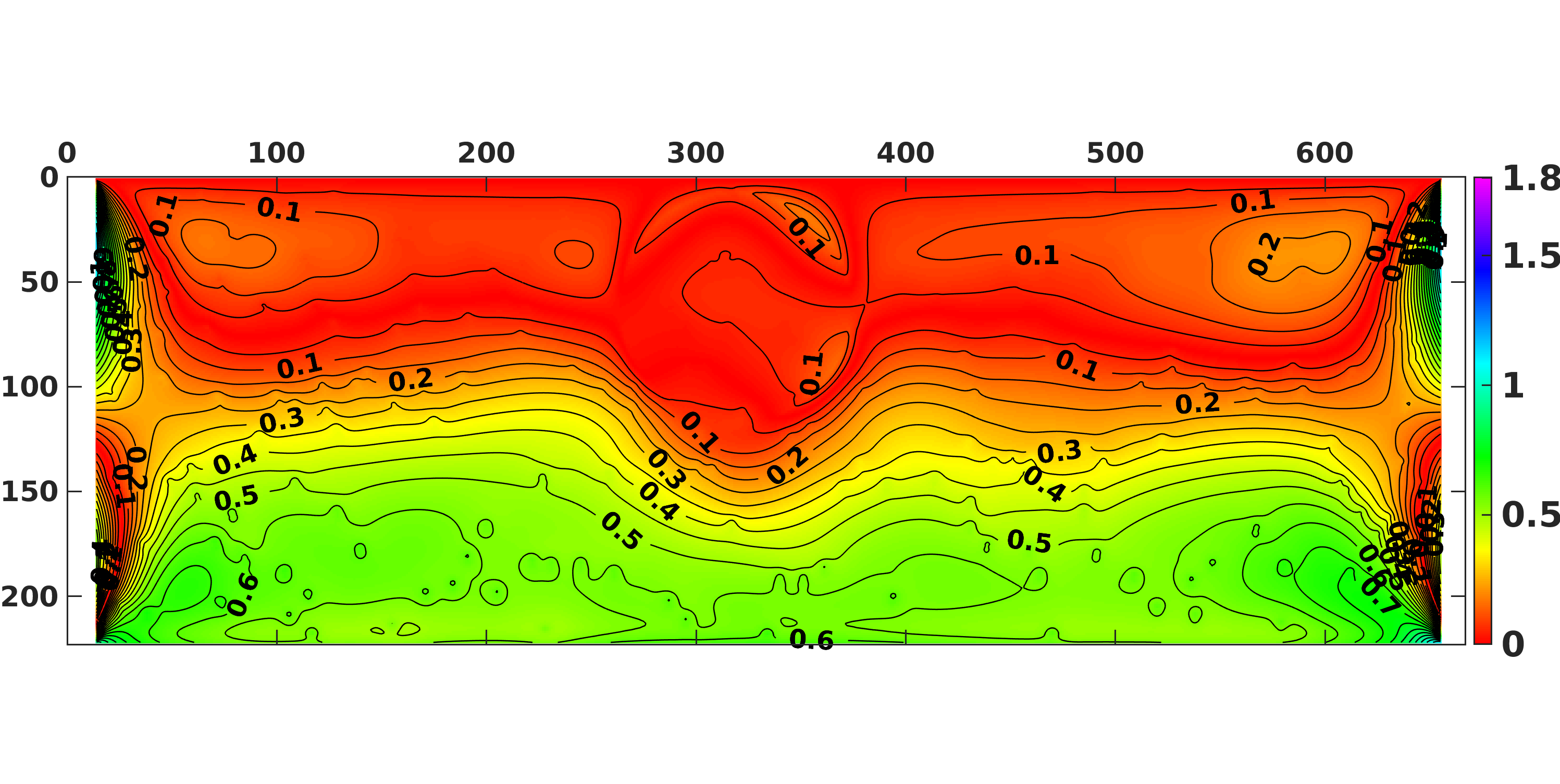}
    \caption{Test 1. Components $u_1$ and $u_2$ (top and bottom left) of estimated displacement field using \eqref{def_J_bS} with $\alpha=0.8$, $\beta=0.5$, multi-scale, and their absolute errors (right).}
    \label{fig_sim_exp_0}
\end{figure}

\begin{figure}[t!]
    \includegraphics[width=0.49\textwidth, clip=true, trim={0cm 4cm 0cm 3cm}]{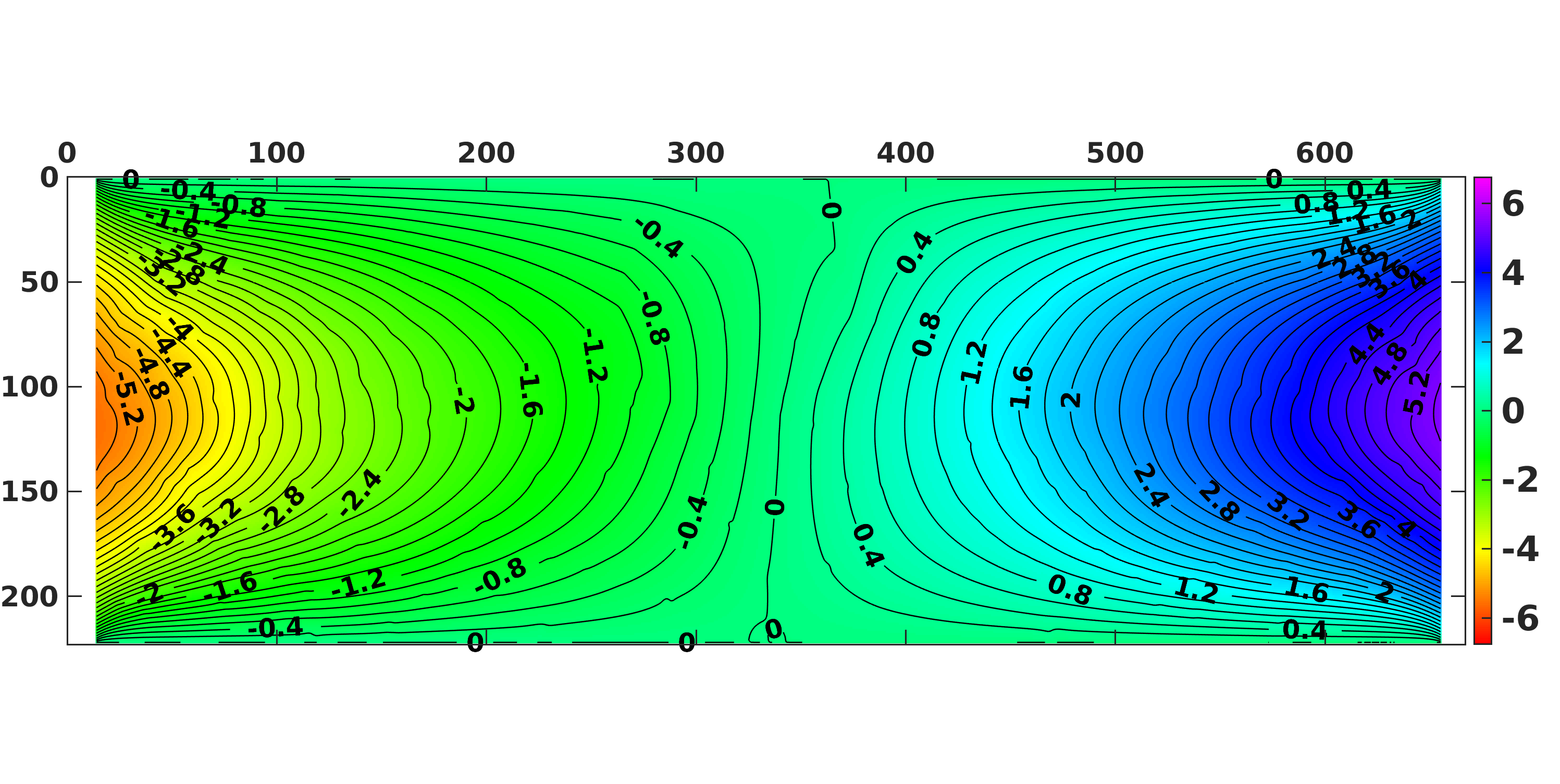}
    \includegraphics[width=0.49\textwidth, clip=true, trim={0cm 4cm 0cm 3cm}]{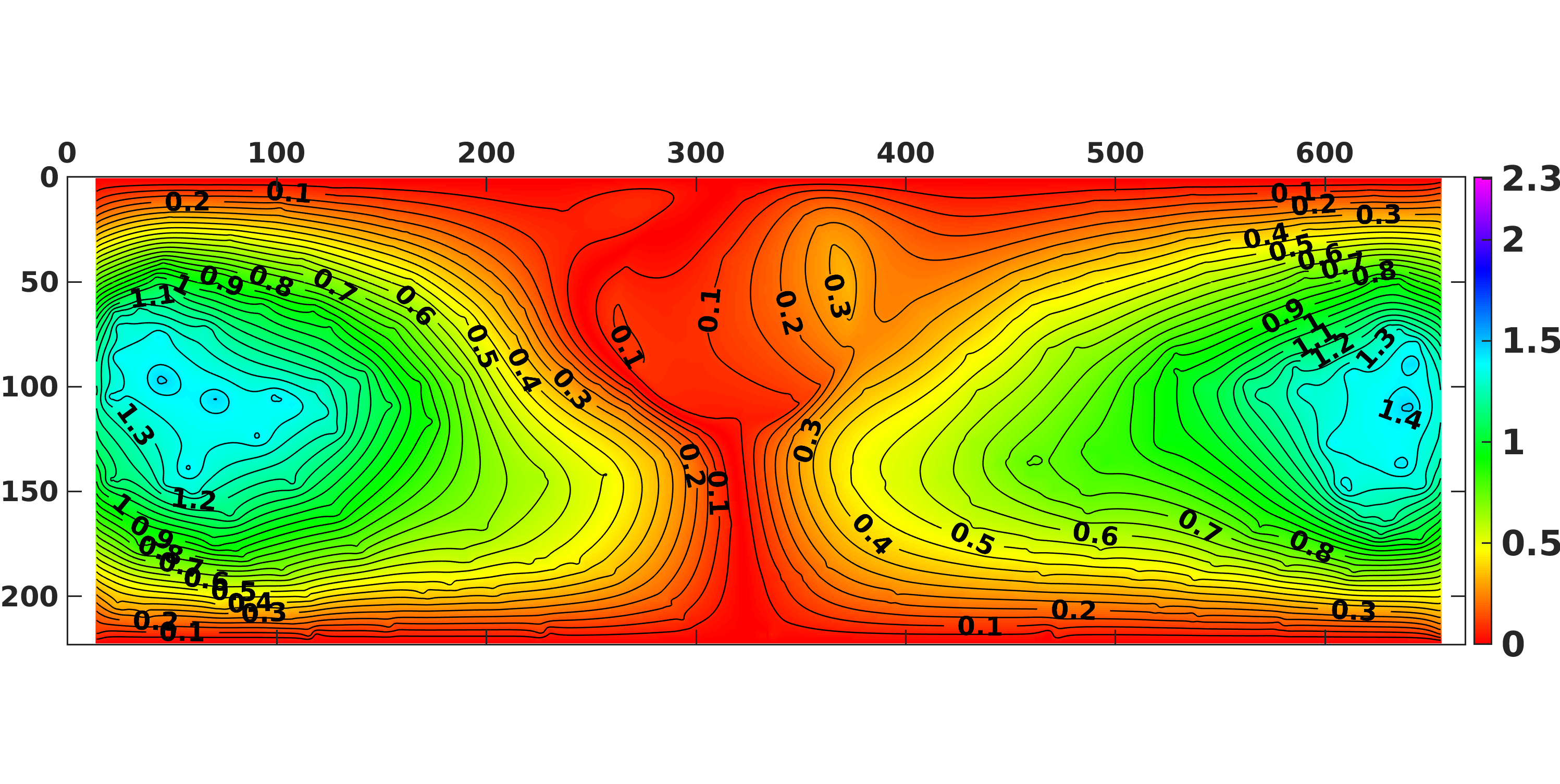}
    \\
    \includegraphics[width=0.49\textwidth, clip=true, trim={0cm 4cm 0cm 3cm}]{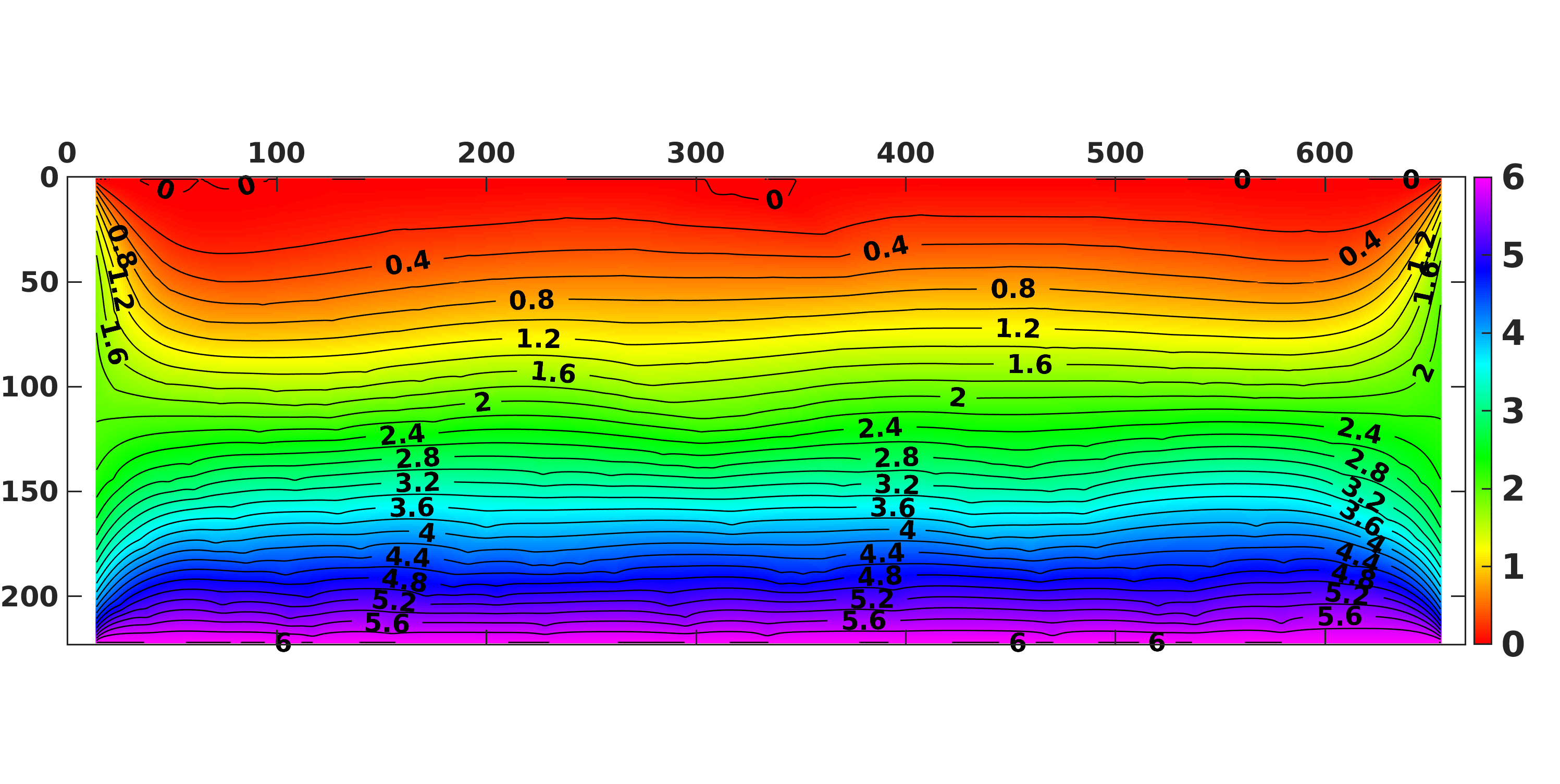}
    \includegraphics[width=0.49\textwidth, clip=true, trim={0cm 4cm 0cm 3cm}]{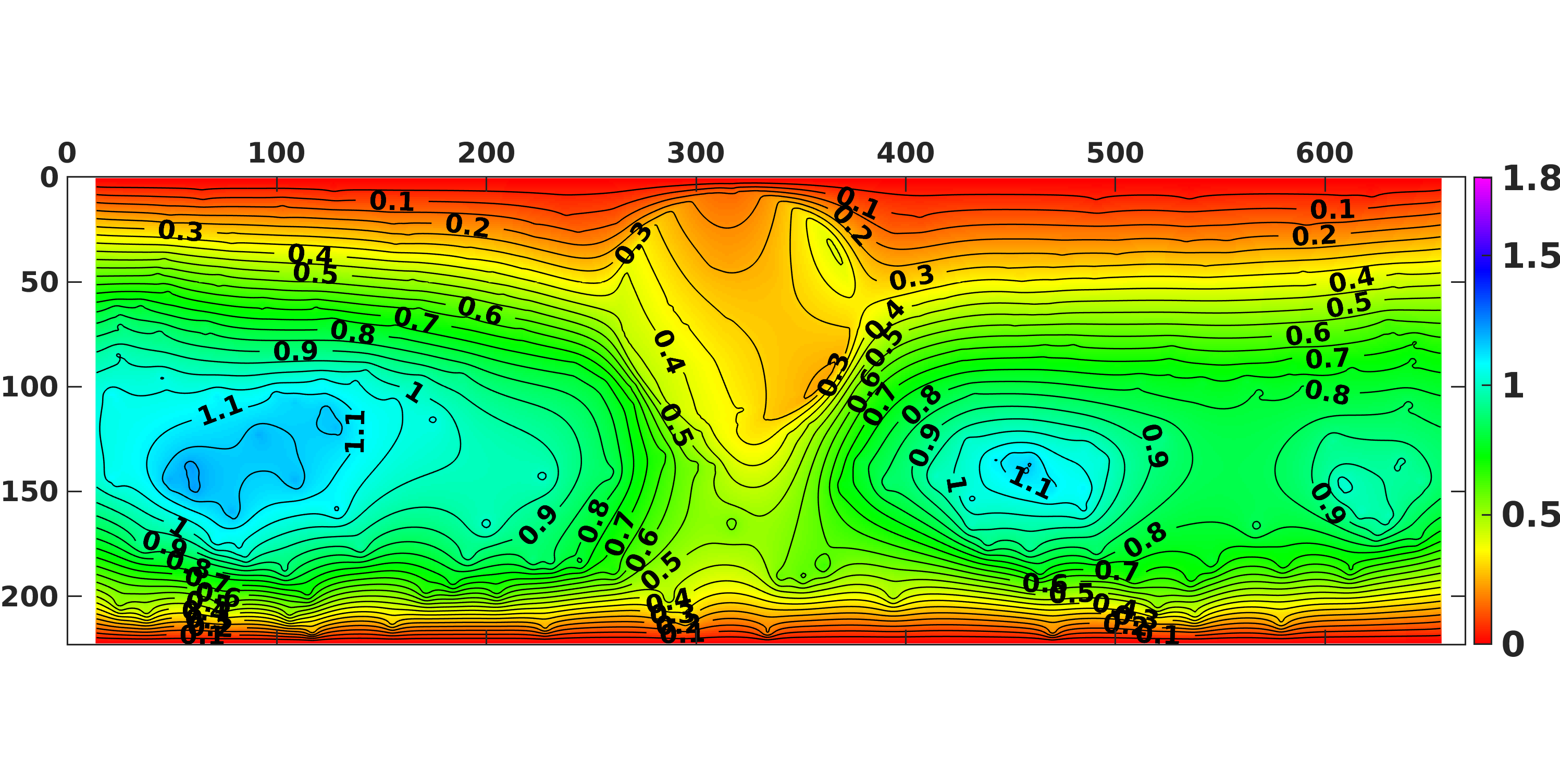}
    \caption{Test 3. Components $u_1$ and $u_2$ (top and bottom left) of estimated displacement field using \eqref{def_F} with $\alpha=0.8$, $\beta=0$, multi-scale, and their absolute errors (right).}
     \label{fig_sim_exp_2}
\end{figure}

\begin{table}[t!]
  \begin{center}
    \begin{tabular}{l|l|r|r|r|r|r|r}
    \hline
      \textbf{Test} & \textbf{Fig.} & \textbf{Parameter} & \textbf{Parameter} & \textbf{Multi-scale} & $e_{rel}(\uf)$ & $e_{rel}(u_1)$ & $e_{rel}(u_2)$\\
      No. & No. & $\alpha$ & $\beta$ & & \% & \% & \% \\
      \hline
      1 & 5 & 0.8 & 0.5 & yes & 19.22 & 28.07 & 10.10 \\
      2 &   & 0.8 & 0   & no  & 31.69 & 35.94 & 28.64 \\
      3 & 6 & 0.8 & 0   & yes & 20.92 & 23.77 & 18.87 \\
      4 &   & 0.8 & 0.5 & no  & 10.21 & 12.08 & 8.81 \\
      5 & 7 & 0.8 & 0.5 & yes & 6.48  & 7.70  & 5.56 \\
      \hline
    \end{tabular}
  \end{center}
  \caption{Relative errors of the estimated fields and their components.}
  \label{tab_rel_err}
\end{table}

\begin{figure}[t!]
    \includegraphics[width=0.49\textwidth, clip=true, trim={0cm 4cm 0cm 3cm}]{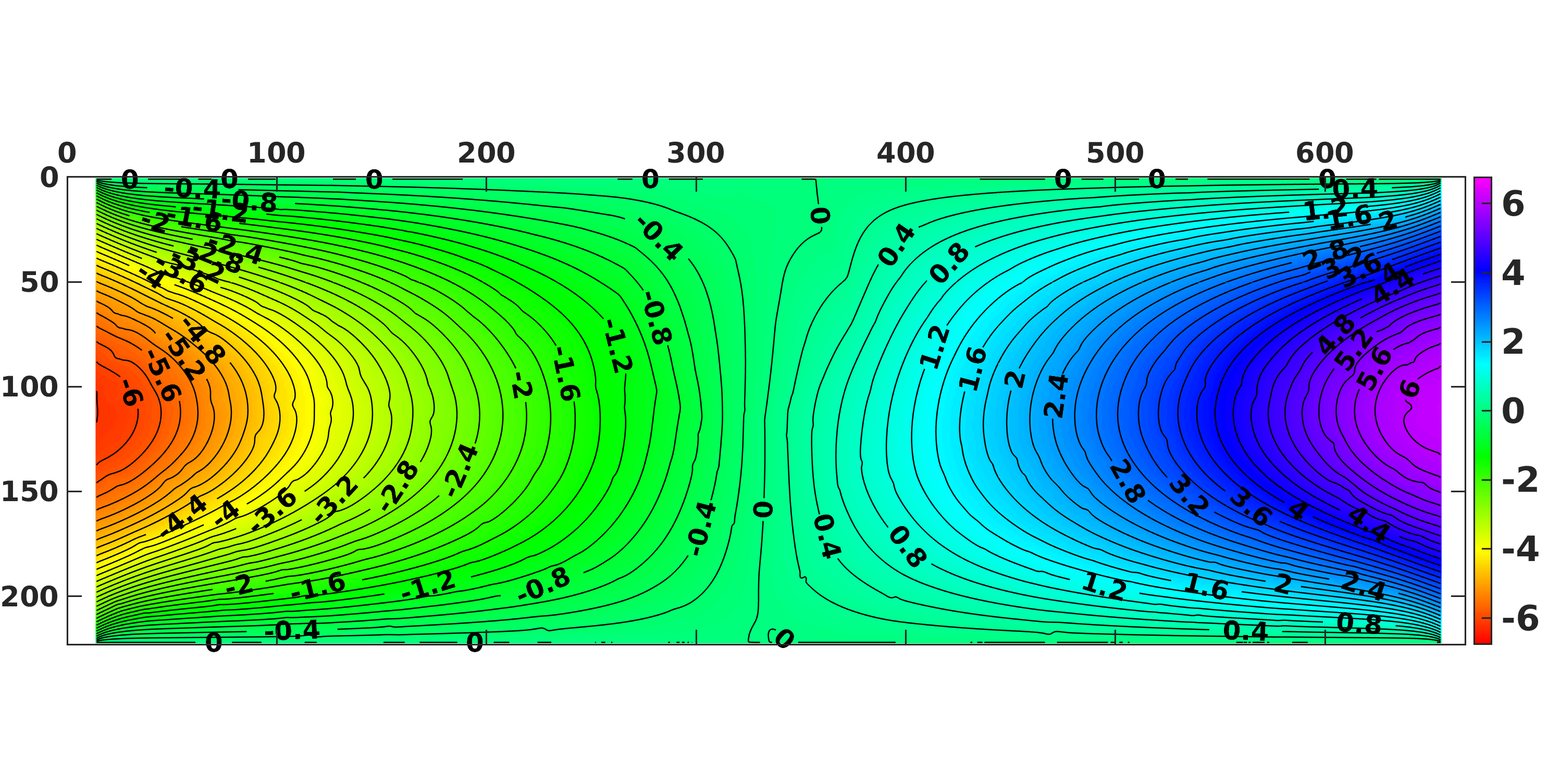}
    \includegraphics[width=0.49\textwidth, clip=true, trim={0cm 4cm 0cm 3cm}]{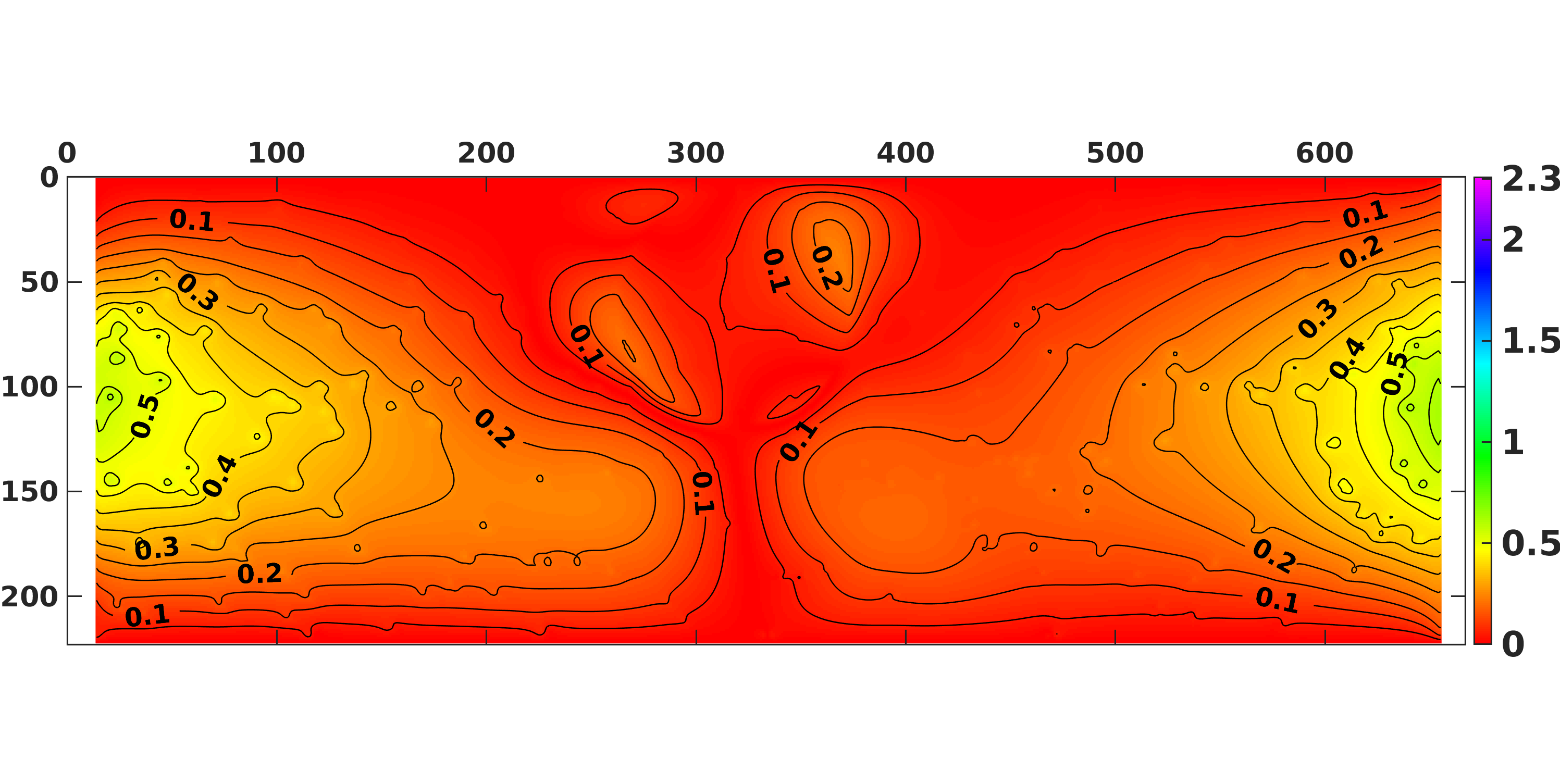}
    \\
    \includegraphics[width=0.49\textwidth, clip=true, trim={0cm 4cm 0cm 3cm}]{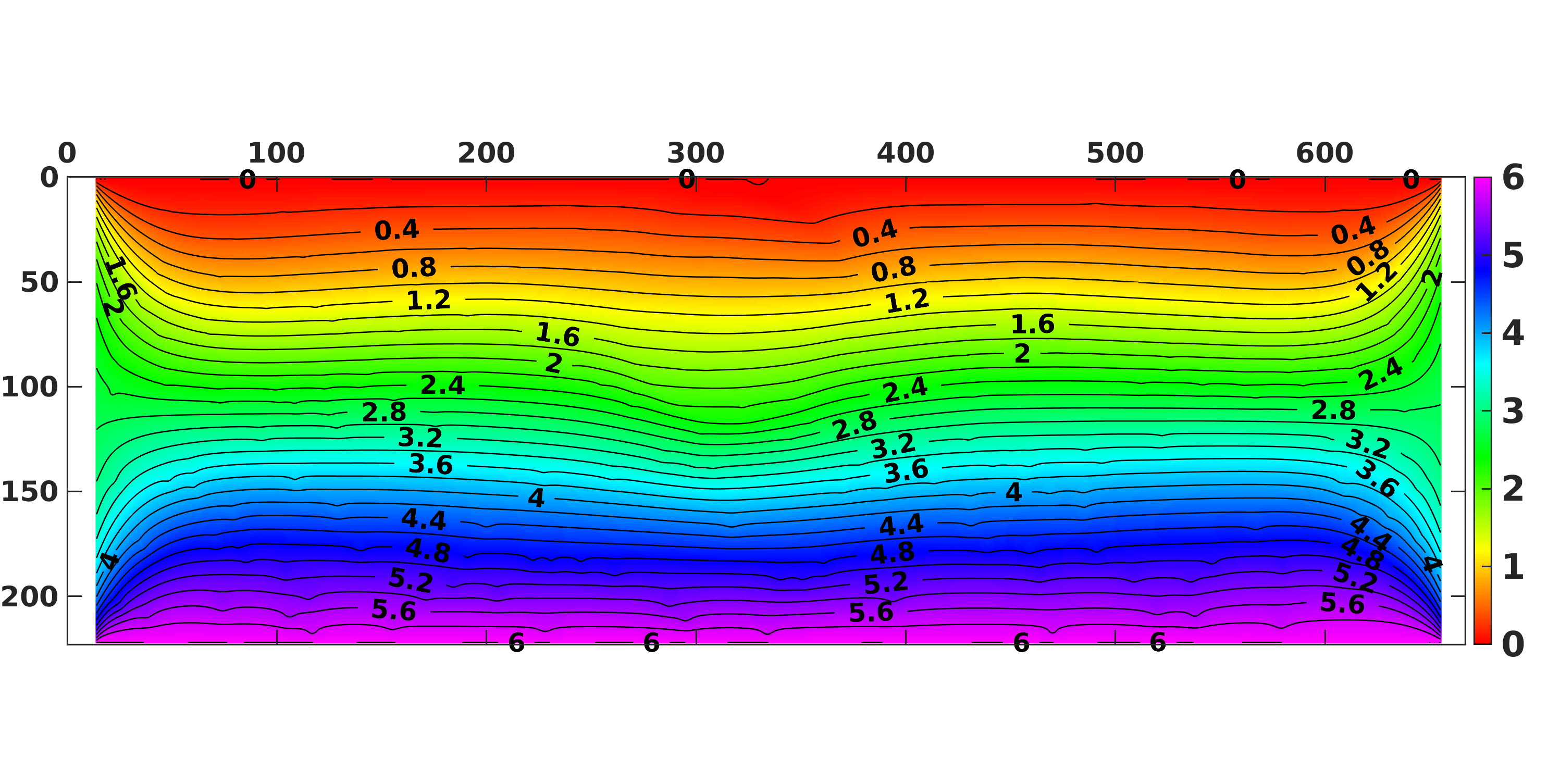}
    \includegraphics[width=0.49\textwidth, clip=true, trim={0cm 4cm 0cm 3cm}]{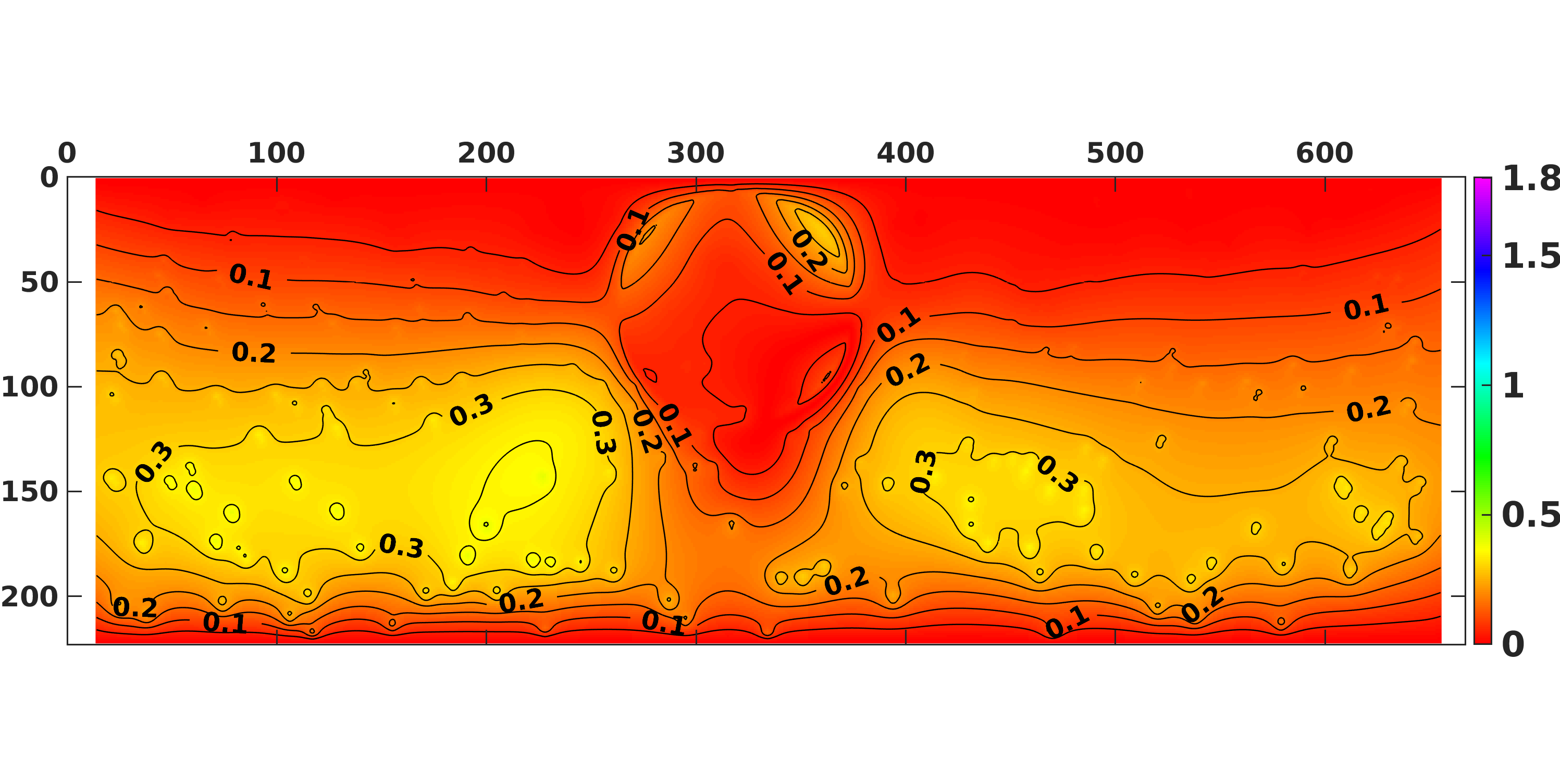}
    \caption{Test 5. Components $u_1$ and $u_2$ (top and bottom left) of estimated displacement field using \eqref{def_F} with $\alpha=0.8$, $\beta=0.5$, multi-scale, and their absolute errors (right).}
    \label{fig_sim_exp_4}
\end{figure}

% Subsection - Experimental Data
\subsection{Experimental Data}

Next, we consider the data from an actual quasi-static elastography experiment. The sample has the same structure as the simulated sample, and its OCT tomograms before and after compression are depicted in Figure~\ref{fig_exp_oct}. The red arrows correspond to the vectors $\ufhi$ obtained by speckle tracking as described in \cite{Sherina_Krainz_Hubmer_Drexler_Scherzer_2020}. 
The resulting field estimate from the standard optical flow is depicted in Figure~\ref{fig_results} (left) and its approximate absolute errors in Figure~\ref{fig_exp_res} (left) in comparison to the expected field for considered sample. The error reaches $5.7$ pixels in $u_1$ and $4.7$ pixel in $u_2$. Figures~\ref{fig_results} and \ref{fig_exp_res} (right) depict the reconstructed field and its errors using the proposed elastographic optic flow method \eqref{def_F} combining the smoothness and speckle regularization terms together with the background information. For this test, we used the same parameter choice as in Test~5. The resulting displacement from the experimental data features only $0.6$ pixel misfit in $u_1$ and $0.8$ pixel in $u_2$ in certain areas in the sample.

\begin{figure}[t!]
    \includegraphics[width=0.49\textwidth, clip=true, trim={0cm 4cm 0cm 3cm}]{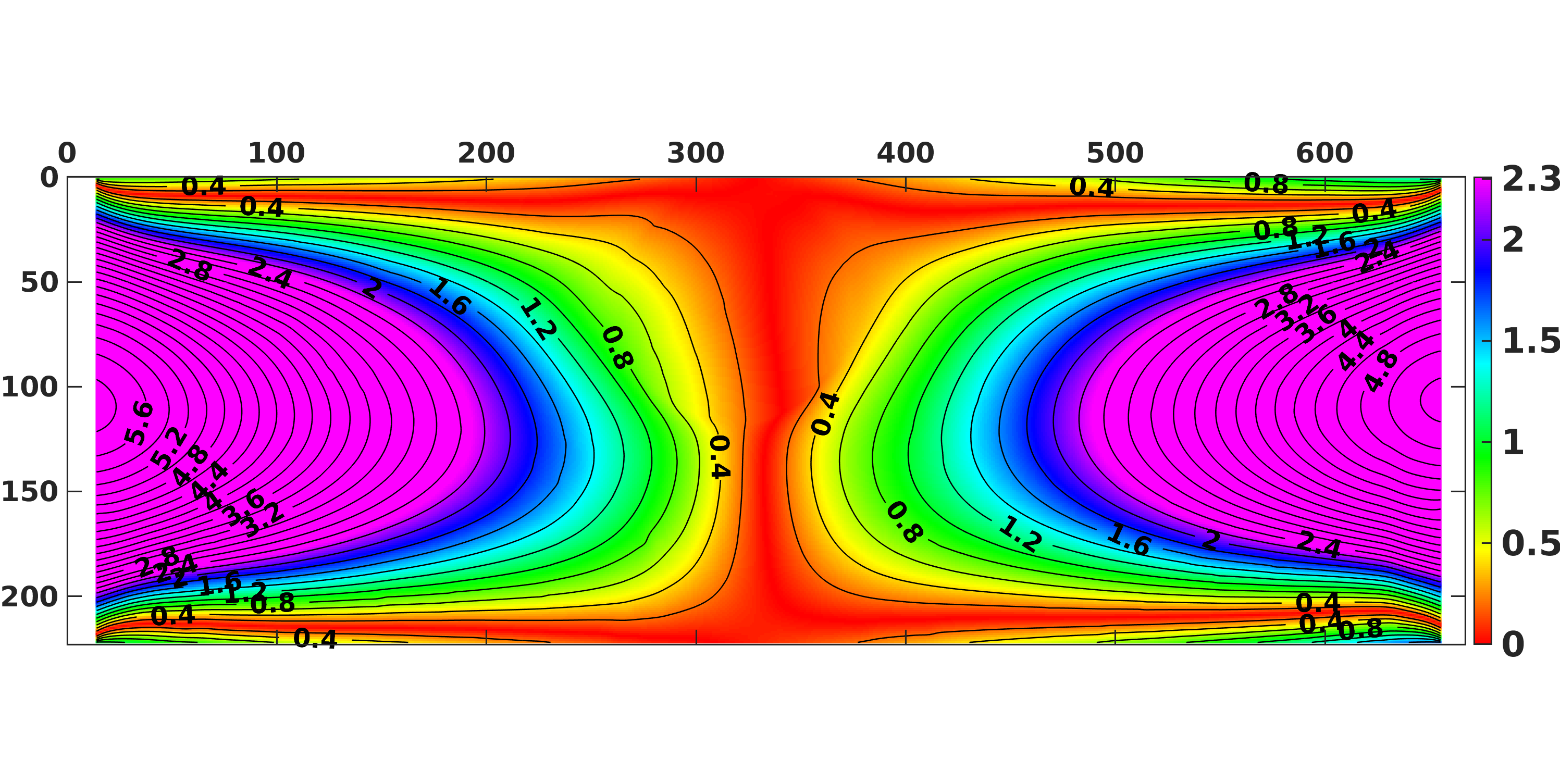}
    \includegraphics[width=0.49\textwidth, clip=true, trim={0cm 4cm 0cm 3cm}]{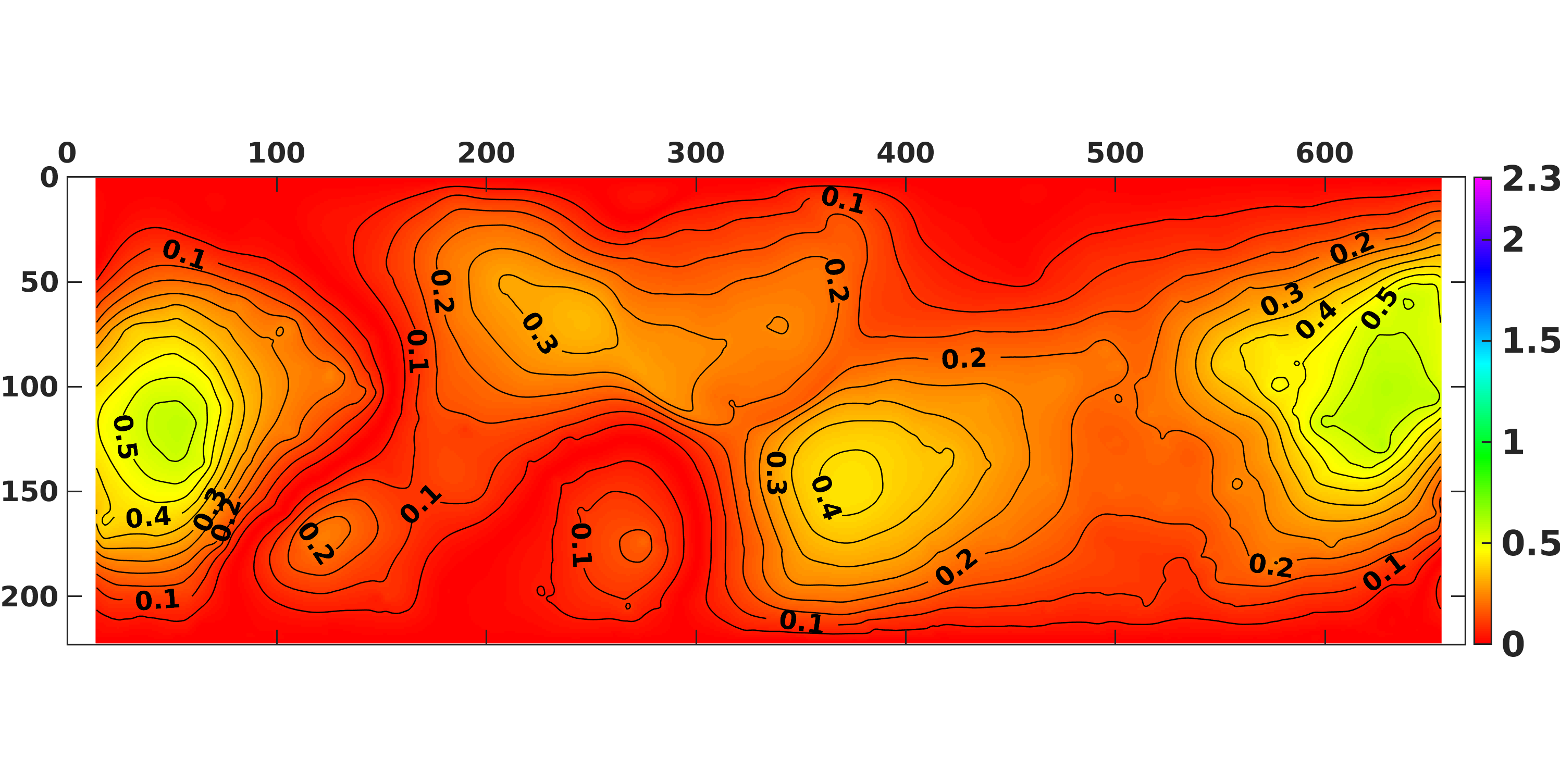}
    \\
    \includegraphics[width=0.49\textwidth, clip=true, trim={0cm 4cm 0cm 3cm}]{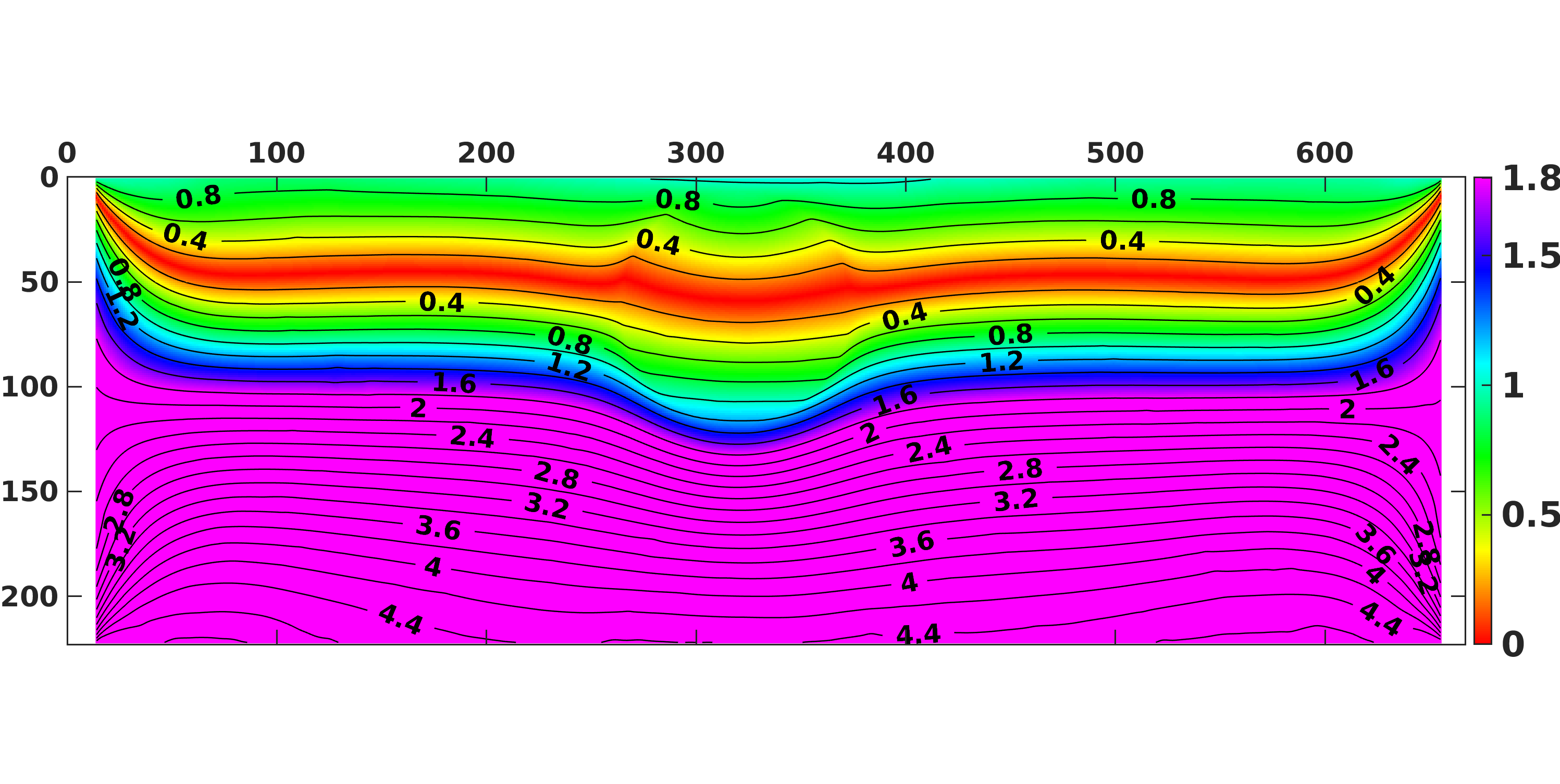}
    \includegraphics[width=0.49\textwidth, clip=true, trim={0cm 4cm 0cm 3cm}]{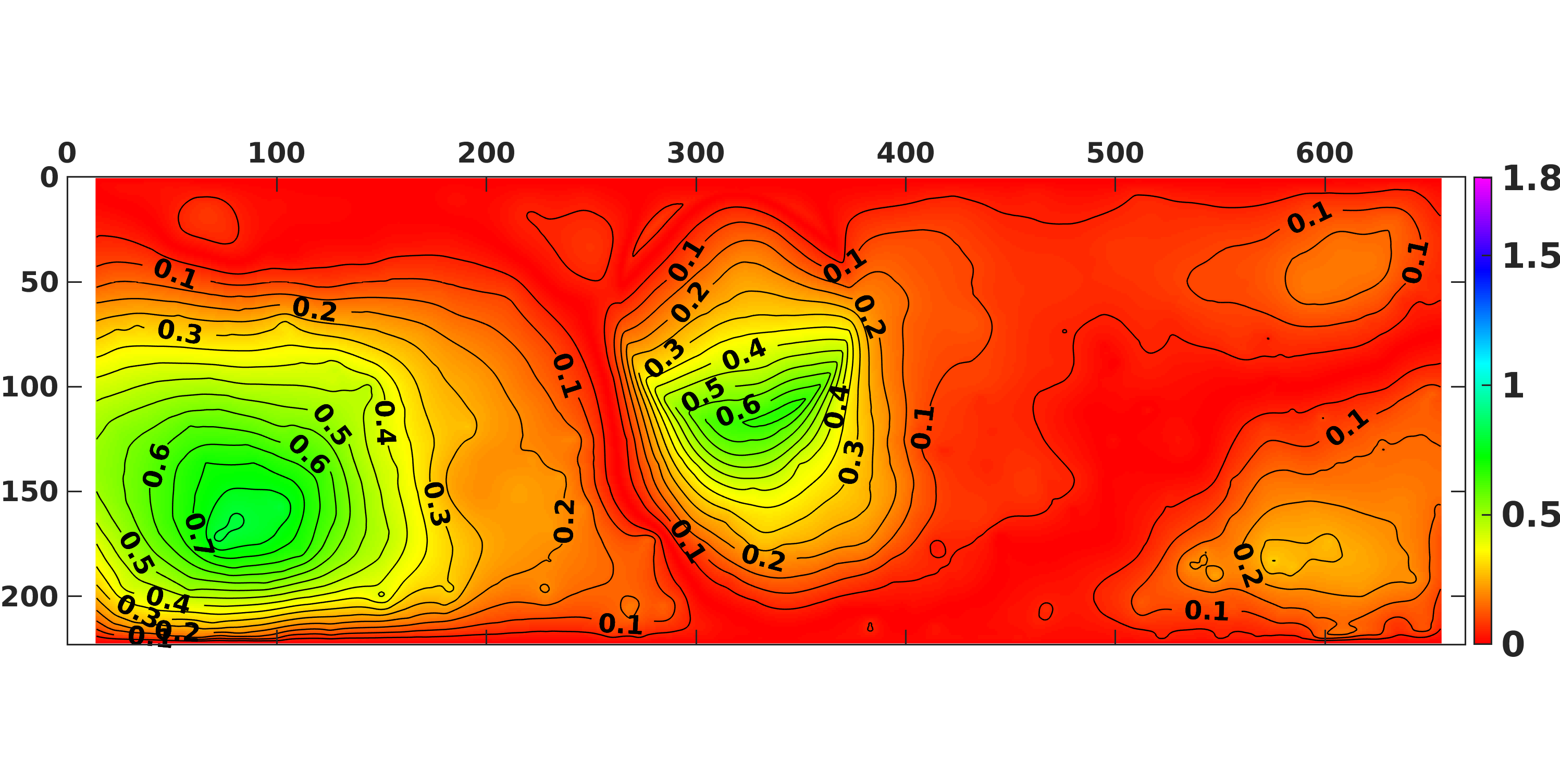}
    \caption{Approximate absolute errors in $u_1$ (top) and $u_2$ (bottom) of estimated fields with the standard optical flow method (left) and EOFM (right).}
    \label{fig_exp_res}
\end{figure}

% % % % % % % % % % %
% Section - Summary %
% % % % % % % % % % %
\section{Summary}

In this paper we introduced an elastographic optical flow method, which allows efficient reconstruction of the displacement field in elastography with scattering imaging experiments. We demonstrate that an efficient algorithm has to take into account information on the physical experiment, such as appropriate boundary conditions, the background medium, and speckle information. The numerical experiments with real data and synthetic data drastically show the necessity of this experimental information.

% % % % % % % % % % % % % %
% Section - Bibliography  %
% % % % % % % % % % % % % %

\end{document}